\documentclass[a4paper,10pt]{amsart}
\usepackage[utf8x]{inputenc}

\usepackage{amssymb}
\usepackage[all,cmtip]{xy}
\usepackage[dvips]{graphicx}
\usepackage{graphics}
\usepackage{calrsfs}
\usepackage{hyperref}
\usepackage{xcolor}

\newtheorem{theorem}{Theorem}[section]
\newtheorem{lemma}[theorem]{Lemma}
\newtheorem{prop}[theorem]{Proposition}
\newtheorem{thm}{Theorem}
\newtheorem{crlr}[theorem]{Corollary}

\theoremstyle{definition}
\newtheorem{rem}[theorem]{Remark}
\newtheorem*{ack*}{Acknowledgements}

\numberwithin{equation}{section}

\newcommand{\grw}{\mathbin{\rotatebox[origin=c]{-45}{$\infty$}}}

\title{On the $\Sigma$-invariants of wreath products}
\author{Luis Augusto de Mendon\c{c}a}
\address{Department of Mathematics, University of Campinas (UNICAMP), rua S\'{e}rgio Buarque de Holanda, 651, 13083-859, Campinas-SP, Brazil.}
\email{luismendonca@ime.unicamp.br}
\begin{document}

\begin{abstract}
We present a full description of the Bieri-Neumann-Strebel invariant of restricted permutational wreath products of groups. We also give partial results
about the $2$-dimensional homotopical invariant of such groups. These results may be turned into a full picture of these invariants when the abelianization of the 
basis group is infinite. We apply these descriptions to the study of the Reidemeister number of automorphisms of wreath products in some specific cases. 
\end{abstract}

\maketitle

\begin{section}{Introduction}

In this paper we study the so called $\Sigma$-invariants of restricted permutational wreath products of groups. The $\Sigma$-invariants of a group are some
subsets of its character sphere and contain a lot of information on finiteness properties of its subgroups. Their definitions and most general results 
appeared in a series of papers by Bieri, Neumann, Strebel, Renz (\cite{BNS},\cite{BieriRenz},\cite{BieriStrebel}) and others.

Let $\Gamma$ be a finitely generated group. The \textit{character sphere} $S(\Gamma)$ is the set of non-zero homomorphisms $\chi: \Gamma \to \mathbb{R}$ (these homomorphisms
are called \textit{characters}) modulo 
the equivalence relation given by $\chi_1 \sim \chi_2$ if there is some $r \in \mathbb{R}_{>0}$ such that $\chi_2 = r \chi_1$. The class of $\chi$ will
be denoted by $[\chi]$. The character sphere may be seen as the $(n-1)$-sphere
in the vector space $Hom(\Gamma, \mathbb{R}) \simeq \mathbb{R}^n$, where $n$ is the torsion-free rank of the abelianization of $\Gamma$. 

In this paper we deal with the homotopical invariants in low dimension, that is, those denoted by $\Sigma^1(\Gamma)$ and $\Sigma^2(\Gamma)$, the second
defined when $\Gamma$ is finitely presented. They are defined as certain subsets of $S(\Gamma)$; we leave the details to Section \ref{sigmabackground}.
Their most important feature is that they 
classify the properties of being finitely generated and being finitely presented for subgroups of $\Gamma$ containing the derived subgroup $[\Gamma, \Gamma]$
(see Theorem \ref{thmfund}).

Recall that a group $\Gamma$ is of type $F_n$ if there is a $K(\Gamma,1)$-complex with compact $n$-skeleton. A group is of type $F_1$ (resp. $F_2$) if 
and only if it is finitely generated (resp. finitely presented). The homological version of the property $F_n$ is the property $FP_n$: a group
$\Gamma$ is of type $FP_n$ if the trivial $\mathbb{Z}\Gamma$-module $\mathbb{Z}$ admits a projective resolution
\[\mathcal{P}: \ldots \to P_n \to P_{n-1} \to \ldots \to P_1 \to P_0 \to \mathbb{Z} \to 0\]
with $P_j$ finitely generated for all $j \leq n$. Again a group is of type $FP_1$ if and only if it is finitely generated, but the properties $F_n$
are in general stronger then $FP_n$. In particular, $FP_2$ is strictly weaker then finite presentability (\cite{BestvinaBrady},\cite{Bieri}).

There are some higher homotopical invariants, denoted by $\Sigma^n(\Gamma)$, which are defined for groups of type $F_n$ and fit in a decreasing sequence
\[S(\Gamma) \supseteq \Sigma^1(\Gamma) \supseteq \Sigma^2(\Gamma) \supseteq \ldots \supseteq \Sigma^n(\Gamma) \supseteq \ldots\]
whenever defined. They classify the property $F_n$ for subgroups above the derived subgroup. Similarly, the homological invariants
$\Sigma^n(\Gamma; \mathbb{Z})$ are defined for groups of type $FP_n$ and classify this same property for subgroups
containing the derived subgroup. In general $\Sigma^1(\Gamma) = \Sigma^1(\Gamma; \mathbb{Z})$ if $\Gamma$ is finitely generated and ${	\Sigma}^n(\Gamma) \subseteq \Sigma^n(\Gamma; \mathbb{Z})$
if $\Gamma$ has type $F_n$ (\cite{BieriRenz}).

All these invariants are in general hard to describe for specific groups, and this has been done only for a few classes of groups. For right-angled
Artin groups, for example, the invariant $\Sigma^1$ was computed first by Meier and VanWyk \cite{MV} and then generalized for higher dimensions 
(for both homotopical and homological versions) by the same authors and Meinert \cite{MMV}. This is connected with the existence of subgroups of these groups having a wide variety
of finiteness properties, as shown by Bestvina and Brady \cite{BestvinaBrady}. Another line of generalization was followed by Meinert, who computed the invariants
in dimension $1$ for graph products \cite{Meinert2}.

Another interesting group for which the invariants are known is Thompson's group $F$. Both homological and homotopical invariants have been computed in all dimensions
by Bieri, Geoghegan and Kochloukova \cite{BGK}. The $\Sigma^2$-invariants of the generalized Thompson groups $F_{n, \infty}$ were then computed by Kochloukova
\cite{Kochloukova2} and recently Zaremsky extended it to higher dimensions \cite{Zaremsky}.

We considered the homotopical invariants $\Sigma^1$ and $\Sigma^2$ of wreath products. Recall that given $H$ and $G$ groups and a $G$-set 
$X$, the wreath product $H \wr_X G$ is defined as the semi-direct product $M \rtimes G$, where $M = \oplus_{x \in X} H_x$ is the direct sum (that is, the 
restricted direct product) of copies of $H$ indexed by $X$ and $G$ acts by permuting this copies according to the action on $X$. We shall always assume that
$X \neq \emptyset$ and $H \neq 1$, to avoid trivial cases. The finiteness properties
for these groups were studied by Cornulier in \cite{Cornulier} (finite generation and finite presentability) and more recently by Bartholdi, Cornulier and Kochloukova 
\cite{BCK} (properties $FP_n$ and $F_n$).

\begin{rem}
By $H_x$ we always mean the copy of $H$ associated to the element $x \in X$. On the other hand, $G_x$ denotes the stabilizer of $x \in X$ in the action 
of $G$. To avoid confusion, we will always denote by $G$ the group that acts.
\end{rem}

Our first result is the full description of $\Sigma^1$.

\begin{thm}  \label{thmSigma1}
Let $\Gamma = H \wr_X G$ be a finitely generated wreath product and let $\chi: \Gamma \to \mathbb{R}$ be a non-trivial character. We set 
$M = \oplus_{x \in X} H_x \subseteq \Gamma$.
\begin{enumerate}
 \item If $\chi |_M = 0$, then $[\chi] \in \Sigma^1(\Gamma)$ if and only if $[\chi |_G] \in \Sigma^1(G)$ and $\chi |_{G_x} \neq 0$ for all $x \in X$.
 \item If $\chi |_M \neq 0$, then $[\chi] \in \Sigma^1(\Gamma)$ if and only if at least one of the following conditions holds:
        \begin{itemize}
         \item[(a)] There exist $x,y \in X$ with $x \neq y$, $\chi |_{H_x} \neq 0$ and $\chi |_{H_y} \neq 0$;
         \item[(b)] There exists $x \in X$ with $\chi |_{H_x} \neq 0$ and $[\chi |_{H_x}] \in \Sigma^1(H)$ or
         \item[(c)] $\chi |_{G} \neq 0$.
         \end{itemize} 
\end{enumerate}
\end{thm}

Part $(1)$ of the above theorem generalizes Theorem $8.1$ in \cite{BCK} in dimension $1$, where $H$ has infinite abelianization by hypothesis. For regular
wreath products, that is, $\Gamma = H \wr_G G$, the action being by multiplication on the left, the $\Sigma^1$-invariant was already computed by Strebel in
Proposition C1.18 in \cite{Strebel}.

For the invariant $\Sigma^2$ we consider two cases, the same as in the theorem above. For characters $\chi: H \wr_X G \to \mathbb{R}$ such that $\chi |_M \neq 0$
the criteria developed by Renz \cite{Renz} are especially powerful, and have allowed us to prove part $(2)$ of Theorem \ref{thmSigma1} and a similar result 
for $\Sigma^2$. 

\begin{thm}   \label{thmB}
Let $\Gamma = H \wr_X G$ be a finitely presented wreath product and let $\chi: \Gamma \to \mathbb{R}$ be a non-trivial character. If the set
 \[ T = \{x \in X \hbox{ } | \hbox{ } \chi |_{H_x} \neq 0\}\]
has at least $3$ elements, then $[\chi] \in \Sigma^2(\Gamma)$.  
\end{thm}

The cases where $T$ is non-empty but has less than $3$ elements can be dealt with using the direct product formula (see Theorem \ref{GehrkeThm}) and the results on the $\Sigma^1$-invariant 
(see Theorem \ref{sigma2viarenz} and the comment right before it).

For the characters $\chi: \Gamma \to \mathbb{R}$ with $\chi |_M = 0$ we were not able to obtain a complete result, by lack of a general method to study
necessary conditions for $[\chi] \in \Sigma^2(\Gamma)$. By the results of Bartholdi, Cornulier and Kochloukova on homological invariants, the most general
theorem we can enunciate is the following, where $G_{(x,y)}$ denotes the stabilizer subgroup associated to an element $(x,y)$ of $X^2$, which is equipped with
the diagonal $G$-action.
 
 \begin{thm}  \label{thmC}
 Let $\Gamma = H \wr_X G$ be a finitely presented wreath product and let $\chi: \Gamma \to \mathbb{R}$ be a non-zero character such that $\chi |_M = 0$. 
 Then $[\chi] \in \Sigma^2(\Gamma)$ if all three conditions below hold
 \begin{enumerate}
    \item $[\chi |_G] \in \Sigma^2(G)$;
    \item $[\chi |_{G_x}] \in \Sigma^1(G_x)$ for all $x \in X$ and
    \item $\chi |_{G_{(x,y)}} \neq 0$ for all $(x,y) \in X^2$.
      \end{enumerate}
In general, conditions $(1)$ and $(3)$ are necessary for $[\chi] \in \Sigma^2(\Gamma)$. If we assume further that
the abelianization of $H$ is infinite, then condition $(2)$ is necessary as well.
 \end{thm}
 
 Restrictions on the abelianization of the basis group $H$ have been recurrent in the study of finiteness properties of wreath products and related 
 constructions. Besides appearing in the work of Bartholdi, Cornulier and Kochloukova \cite{BCK}, they also pop up in the paper by Kropholler and Martino \cite{KrophollerMartino}, 
 which deals with the wider class of \textit{graph-wreath products} (see Section \ref{gwp}) from a more homotopical point of view.
 
 Finally, we consider some applications to twisted conjugacy. Recall that given an automorphism $\varphi$ of a group $G$, the \textit{Reidemeister number} $R(\varphi)$
 is defined as the number of orbits of the twisted conjugacy action, which is given by $g \cdot h := g h \varphi(g^{-1})$, for $g,h \in G$. 
 
 Exploring the connections between $\Sigma$-theory and 
 Reidemeister numbers, as found out by Koban and Wong \cite{KobanWong} and Gon\c{c}alves and Kochloukova \cite{GonKoch}, we obtain some results about the Reidemeister numbers
 of automorphisms contained in some subgroups of finite index of $Aut(H \wr_X G)$, under some relatively strong restrictions. For precise statements, see Corollaries 
 \ref{crlr1} and \ref{crlr2}.
 
 \end{section}

\begin{section}{Background on the \texorpdfstring{$\Sigma$}{Sigma}-invariants} \label{sigmabackground}
 Let us start by recalling the definition of the invariant $\Sigma^1$. For a finitely generated group $\Gamma$ and a finite generating set $\mathcal{X} \subseteq \Gamma$, we 
 consider the Cayley graph $Cay(\Gamma; \mathcal{X})$. Its vertex set is $\Gamma$ and two vertices $\gamma_1$ and $\gamma_2$ are connected by and edge if and 
 only if there is some $x \in \mathcal{X}^{\pm 1}$ such that $\gamma_2 = \gamma_1 x$ (therefore $\Gamma$ acts on the left). This graph is always connected. 
 Given a non-zero character $\chi: \Gamma \to \mathbb{R}$ we can define the 
 submonoid 
 \[ \Gamma_{\chi} = \{\gamma \in \Gamma \hbox{ } | \hbox{ } \chi(\gamma) \geq 0\}.\]
 Notice that $\Gamma_{\chi_1} = \Gamma_{\chi_2}$ if and only if $\chi_1$ and $\chi_2$ represent the same class in the character sphere $S(\Gamma)$.
 The full subgraph spanned by $\Gamma_{\chi}$, which we denote by $Cay(\Gamma; \mathcal{X})_{\chi}$, may not be connected. We put:
  \[ \Sigma^1(\Gamma) = \{ [\chi] \in S(\Gamma) \hbox{ } | \hbox{ } Cay(\Gamma; \mathcal{X})_{\chi} \hbox{ is connected} \}.\]
 It can be shown that this definition does not depend on the (finite) generating set $\mathcal{X}$. This invariant is known as the \textit{Bieri-Neumann-Strebel}
 invariant (or simply BNS-invariant), in reference to the authors who studied it first \cite{BNS}.
 
 The invariant $\Sigma^2$ is defined similarly. If $\Gamma$ is finitely presented and $\langle \mathcal{X} | \mathcal{R} \rangle$ is a finite presentation,
 we consider the Cayley complex $Cay(\Gamma;  \langle \mathcal{X} | \mathcal{R} \rangle)$. This complex is obtained from the Cayley graph by gluing 
 $2$-dimension cells with boundary determined by the loops defined by the relations $r \in R$, for each base point in $\Gamma$. The resulting complex is always
 $1$-connected. Again we define
 $Cay(\Gamma;  \langle \mathcal{X} | \mathcal{R} \rangle)_{\chi}$ to be the full subcomplex spanned by $\Gamma_{\chi}$. 
 The $1$-connectedness of this complex depends on the choice of the presentation. We
 define $\Sigma^2(\Gamma)$ as the subset of $S(\Gamma)$ containing exactly all the classes $[\chi]$ of characters such that 
 $Cay(\Gamma; \langle \mathcal{X} | \mathcal{R} \rangle)_{\chi}$ is $1$-connected for some finite presentation $
 \langle \mathcal{X} | \mathcal{R} \rangle$ of $\Gamma$. More details on these definitions may be found in \cite{Meinert}.
 
 The main feature of these invariants is that they classify the related finiteness properties for subgroups containing the derived subgroup. For the invariants
 $\Sigma^1$ and $\Sigma^2$, this can be stated as follows.

\begin{theorem}[\cite{BNS}, \cite{RenzThesis}] \label{thmfund}
Suppose that $\Gamma$ is finitely generated and let $N \subseteq \Gamma$ be a subgroup such that $[\Gamma, \Gamma] \subseteq N$. Then $N$ is finitely generated
if and only if 
\[ \Sigma^1(\Gamma) \supseteq \{[\chi] \in S(\Gamma) \hbox{ } | \hbox{ } \chi |_N = 0\}.\]
Similarly, if $\Gamma$ is further finitely presented then $N$ is finitely presented
if and only if 
\[ \Sigma^2(\Gamma) \supseteq \{[\chi] \in S(\Gamma) \hbox{ } | \hbox{ } \chi |_N = 0\}.\]
\end{theorem}

 The homological invariants can be defined by means of the monoid ring $\mathbb{Z}\Gamma_{\chi}$. This is of course the subring of $\mathbb{Z}{\Gamma}$ containing
 exactly all elements $\sum a_{\gamma} \gamma \in \mathbb{Z}{\Gamma}$ such that $a_{\gamma} \neq 0$ only if $\gamma \in \Gamma_{\chi}$. We put 
 \[ \Sigma^m(\Gamma; \mathbb{Z}) = \{[\chi] \in S(\Gamma) \hbox{ } | \hbox{ } \mathbb{Z} \hbox{ is of type } FP_m \hbox{ over } \mathbb{Z}\Gamma_{\chi} \}.\]
 As observed by Bieri and Renz \cite{BieriRenz} if $\Sigma^m(\Gamma; \mathbb{Z}) \neq \emptyset$ then $\Gamma$ is of type $FP_m$.
 All we need about these homological invariants is that $\Sigma^2(\Gamma) \subseteq \Sigma^2(\Gamma; \mathbb{Z})$ whenever $\Gamma$ is finitely presented. Details may be found
 in \cite{BieriRenz} and \cite{RenzThesis}.
 
 Some of the general results we will need about these invariants concern direct products of groups, subgroups of finite index and retracts. 
 
\begin{theorem}[Direct product formulas, \cite{Gehrke}] \label{GehrkeThm}
Let $G_1$ and $G_2$ be finitely generated groups and let $\chi = (\chi_1, \chi_2) : G_1 \times G_2 \to \mathbb{R}$ be a non-zero character. Then 
$[\chi] \in \Sigma^1(G_1 \times G_2)$ if and only if at least one of the following conditions holds:
\begin{enumerate}
 \item $\chi_i \neq 0$ for $i=1,2$ or
 \item $[\chi_i] \in \Sigma^1(G_i)$ for some $i \in \{1,2\}$.
 \end{enumerate}
Similarly, if $G_1$ and $G_2$ are finitely presented, then $[\chi] \in \Sigma^2(G_1 \times G_2)$ if and only if at least one of the following conditions holds:
\begin{enumerate}
 \item $[\chi_1] \in \Sigma^1(G_1)$ and $\chi_2 \neq 0$; 
\item $[\chi_2] \in \Sigma^1(G_2)$ and $\chi_1 \neq 0$ or
\item $[\chi_i] \in \Sigma^2(G_i)$ for some $i \in \{1,2\}$.
 \end{enumerate}
\end{theorem}

There was a conjecture suggesting how to compute the $\Sigma$-invariants of direct products in
higher dimensions, but it turned out to be false. Counterexamples were found by Meier, Meinert and VanWyk \cite{MMV2} for the homotopical 
invariants and by Schütz \cite{Schutz} in the homological case. For precise statements see \cite{BieriGeoghegan}, which also brings
a proof of the homological conjecture if coefficients are taken in a field (rather than $\mathbb{Z}$).

\begin{theorem}[Finite index subgroups, \cite{Meinert}]  \label{finiteindexthm}
Let $G$ be a finitely presented group and let $H \leqslant G$ be a subgroup of finite index. 
Let $\chi: G \to \mathbb{R}$ be a non-zero character and denote by $\chi_0$ its restriction to $H$.
Then $[\chi] \in \Sigma^2(G)$ if and only if $[\chi_0] \in \Sigma^2(H)$.
\end{theorem}

\begin{theorem}[Retracts, \cite{Meinert}]  \label{retractsthm}
Let $G$ be a finitely presented group and suppose that $H$ is a retract, that is, there are homomorphisms $p: G \to H$ and $j: H \to G$ such that $p \circ j = id_H$.
Suppose that $\chi: H \to \mathbb{R}$ is a non-zero character. Then
\[ [\chi \circ p] \in \Sigma^2(G) \Rightarrow [\chi] \in \Sigma^2(H).\]
\end{theorem}

\begin{theorem}[Theorem C, \cite{Kochloukova}]  \label{ThmKochloukova}
Suppose that $G$ is a group of type $FP_m$ (resp. $F_m$) and $N$ is a normal subgroup of $G$ that is locally nilpotent-by-finite. Then
\[ \{[\chi] \in S(G) \hbox{ } | \hbox{ } \chi(N) \neq 0 \} \subseteq \Sigma^m(G;\mathbb{Z}) \hbox{ (resp. } \Sigma^m(G)).\] 
\end{theorem}

As pointed out to me by D. Kochloukova, in \cite{Kochloukova} the result is stated for $N$ locally polycyclic-by-finite, but actually the proof works for nilpotent-by-finite. We will use it
with $N$ being abelian. The case $m=1$, with $N$ abelian, can also be found as Lemma C1.20 in Strebel's notes \cite{Strebel}. 
\end{section}

\begin{section}{The \texorpdfstring{$\Sigma^1$}{Sigma1}-invariant of wreath products}  \label{sw}
Let $\Gamma = H \wr_X G$ be a finitely generated wreath product. As shown by Cornulier \cite{Cornulier}, both $G$ and $H$ are finitely generated and $G$ acts on $X$ with 
finitely many orbits. Denote $M = \oplus_{x \in X} H_x \subseteq \Gamma$.
We start working with the characters $\chi: \Gamma \to \mathbb{R}$ such that $\chi |_M = 0$, for which there are some partial results by Bartholdi, 
Cornulier and Kochloukova. We quote their result in its most general form, which deals with the higher homological invariants.

    \begin{theorem}[\cite{BCK}, Theorem 8.1] \label{BCKThm}
Let $\Gamma = H \wr_X G$ be a wreath product of type $FP_m$ and let $M = \oplus_{x \in X}H_x \subseteq \Gamma$. Let $\chi: \Gamma \to \mathbb{R}$ be a non-zero
character such that $\chi |_M = 0$. The following conditions are sufficient for $[\chi] \in \Sigma^m(\Gamma; \mathbb{Z})$:

\begin{enumerate}
 \item $[\chi |_G] \in \Sigma^m(G; \mathbb{Z})$;
 \item $[\chi |_{G_{\alpha}}] \in \Sigma^{m-i}(G_{\alpha}; \mathbb{Z})$ for all stabilizers $G_{\alpha}$ of the diagonal action of $G$ on $X^i$
        and for all $1 \leq i \leq m$.
\end{enumerate}
Moreover, if the abelianization of $H$ is infinite, then such conditions are also necessary.
    \end{theorem}
 
Notice that item $2$ contains a statement about invariants in dimension $0$. For any finitely generated group $V$ and $\chi: V \to \mathbb{R}$,
the condition $[\chi] \in \Sigma^0(V;\mathbb{Z})$ amounts to saying that $\chi$ is a non-zero homomorphism.

Recall that the homological and homotopical invariants coincide in dimension $1$, that is, $\Sigma^1(V;\mathbb{Z}) = \Sigma^1(V)$ 
whenever $V$ is a finitely generated group (see Corollary C1.5, \cite{Strebel}, for instance). It is worth mentioning that if we consider the original
definitions of the invariants in \cite{BNS} and \cite{BieriRenz}, we get that actually the sets $\Sigma^1(V)$
and $\Sigma^1(V; \mathbb{Z})$ are antipodal in $S(V)$, that is, $\Sigma^1(V;\mathbb{Z}) = -\Sigma^1(V)$. This happens because
in \cite{BieriRenz} the authors chose to work with left group actions, while in \cite{BNS} the actions are on the right.
The sign disappears if the choice is consistent.

We can now extract from Theorem \ref{BCKThm} a set of sufficient
conditions for $[\chi] \in \Sigma^1(\Gamma)$. Namely:
 
     \begin{prop}  \label{pro1}
Let $\Gamma = H \wr_X G$ be a finitely generated wreath product and let $\chi: \Gamma \to \mathbb{R}$ be a non-zero character such that $\chi |_M = 0$. 
If $[\chi |_G] \in \Sigma^1(G)$ and if $\chi |_{G_x} \neq 0$ for all stabilizers $G_x$ of the action of $G$ on $X$, then $[\chi] \in \Sigma^1(\Gamma)$.
     \end{prop}
     
\begin{rem}
This conditions could also be obtained by considering an action of $\Gamma$ on a sufficiently nice complex. We shall
apply this reasoning in the study of the invariant $\Sigma^2(H \wr_X G)$.
\end{rem}

This set of conditions is in fact necessary. First, if $\chi: \Gamma \to \mathbb{R}$ and $M \subseteq ker(\chi)$, then
\[ [\chi] \in \Sigma^1(\Gamma) \Rightarrow   [\chi |_G] \in \Sigma^1(G),\]
since $\chi |_G$ coincides with the character $\bar{\chi}$ induced on the quotient $ \Gamma / M \simeq G$ (see \cite{Strebel} Proposition A4.5).

It suffices then to analyze the restriction of $\chi$ to the stabilizer subgroups under the hypothesis that $[\chi] \in \Sigma^1(\Gamma)$.

    \begin{prop}  \label{pro2}
If $[\chi] \in \Sigma^1(\Gamma)$ and $\chi |_M = 0$, then $\chi |_{G_x} \neq 0$ for all $x \in X$.
     \end{prop}
 
     \begin{proof}
Let $X = G \cdot x_1 \sqcup \ldots \sqcup G \cdot x_n$. We only need to show that $\chi |_{G_{x_i}} \neq 0$ for all $i$. By taking the quotient by
$M' = \bigoplus_{x \in X \smallsetminus G \cdot x_i} H_x$, we may assume that $n=1$,
 that is, we consider wreath products of the form $\Gamma = H \wr_X G$ with $X = G \cdot x_1$.
 
Let $Y$ and $Z$ be finite generating sets for $H$ and $G$, respectively. Since $X = G \cdot x_1$ it is clear that $Y \cup Z$ is a finite generating set 
for $\Gamma$ (we see $Y$ as a subset of the copy $H_{x_1}$). Then $Cay(\Gamma; Y \cup Z)_{\chi}$ must be connected, since $[\chi] \in \Sigma^1(\Gamma)$ by hypothesis. 

First, we show that $M$ can be generated by the left conjugates of elements of $Y^{ \pm 1}$ by elements of $G_{\chi}$. Indeed if $m \in M$ there is a path 
in $Cay(\Gamma; Y \cup Z)_{\chi}$ connecting $1$ to $m$, since $m \in M \subseteq ker(\chi) \subseteq \Gamma_{\chi}$. Such a path
has as label a word with letters in $Y^{\pm 1} \cup Z^{\pm 1}$, so we can write:
 \[ m = w_1 v_1 w_2 v_2 \cdots w_k v_k,\]
 where each $w_j$ is a word in $Y^{\pm 1}$ and each $v_j$ is a word in $Z^{\pm 1}$ (possibly trivial). We rewrite:
 \[ m = w_1 (^{v_1}w_2) (^{v_1 v_2}w_3) \cdots  (^{v_1 \cdots  v_{k-1}}w_k) (v_1 \cdots  v_k).\]
Now, $w_1 (^{v_1}w_2) (^{v_1 v_2}w_3) \cdots  (^{v_1 \cdots  v_{k-1}}w_k) \in M$ and $v_1  \cdots  v_k \in G$. But $m \in M$ and $\Gamma = M \rtimes G$, 
so $v_1  \cdots  v_k = 1_G$. Moreover, since $\chi |_M = 0$, it is clear that $\chi(v_1  \cdots  v_j) \geq 0$ for all $1 \leq j \leq k$, so:
  \[ m = w_1 (^{v_1}w_2) (^{v_1 v_2}w_3)  \cdots  (^{v_1  \cdots  v_{k-1}}w_k) \in \langle ^{G_{\chi}}(Y^{\pm 1}) \rangle,\]
  as we wanted.
  
  But then 
\[ M= \langle ^{G_{\chi}}(Y) \rangle \subseteq \langle ^{G_{\chi}}(H_{x_1}) \rangle = \bigoplus_{x \in G_{\chi} \cdot x_1} H_{x_1},\]
that is, $X = G_{\chi} \cdot x_1$. Finally, as $\chi |_G \neq 0$ there is some $g_1 \in G$ 
 such that $\chi(g_1) < 0$. On the other hand, there must be some $g_0 \in G_{\chi}$ such that 
 $g_0 \cdot x_1 = g_1 \cdot x_1$. It follows that $g_1^{-1} g_0 \in G_{x_1}$, with $\chi(g_1^{-1} g_0) = - \chi(g_1) + \chi(g_0) > 0$, 
hence $\chi |_{G_{x_1}} \neq 0$.
     \end{proof}

We obtain part $(1)$ of Theorem \ref{thmSigma1} by combining Propositions \ref{pro1} and \ref{pro2}.
    \end{section}

\begin{section}{The  \texorpdfstring{$\Sigma^1$}{Sigma1}-invariant and Renz's criterion} 
 We shall use the results of Renz \cite{Renz} to consider the characters $\chi: H \wr_X G \to \mathbb{R}$ such that $\chi |_M \neq 0$. Let $\Gamma$ be any
 finitely generated group and let $\mathcal{X} \subseteq \Gamma$ be a finite generating set. For a non-zero character $\chi: \Gamma \to \mathbb{R}$ and for 
 any word $w = x_1 \cdots x_n$, with $x_i \in \mathcal{X}^{\pm 1}$, we denote:
  \[ v_{\chi}(w) := min\{ \chi(x_1 \cdots x_j) \hbox{ }| \hbox{ } 1 \leq j \leq n\}.\]
 
 \begin{theorem}[\cite{Renz}, Theorem 1] \label{Renz1}
 With the notations above,  $[\chi] \in \Sigma^1(\Gamma)$ if and only if there exists $t \in \mathcal{X}^{\pm 1}$
 with $\chi(t)>0$ and such that for all $x \in \mathcal{X}^{\pm 1} \smallsetminus \{t, t^{-1}\}$ the conjugate $t^{-1}xt$ can be represented by a word $w_x$ in $\mathcal{X}^{\pm 1}$
 such that 
 \[v_{\chi}(t^{-1} x t) < v_{\chi} (w_x).\]
 \end{theorem}

 \begin{prop} \label{sigma1viarenz}
  Let $\Gamma = H \wr_X G$ be a finitely generated wreath product and let $[\chi] \in S(\Gamma)$. Suppose that there is some $x_1 \in X$ such that 
  $G \cdot x_1 \neq \{x_1\}$ and $\chi |_{H_{x_1}} \neq 0$. Then $[\chi] \in \Sigma^1(\Gamma)$.
 \end{prop}

 \begin{proof}
  Let  $Y$ and $Z$ be finite generating sets for $H$ and $G$, respectively, and choose $x_1, \ldots, x_n \in X$ such that 
  $X = \bigsqcup_{j=1}^n G \cdot x_j$ (the element $x_1$ is already chosen to satisfy the hypotheses). For each $1 \leq j \leq n$
  let $Y_j$ be a copy of $Y$ inside $H_{x_j}$. It is clear that $\Gamma$ is generated by $Y_1 \cup \ldots \cup Y_n \cup Z$.
  
  Now, since $G \cdot x_1 \neq \{x_1\}$ we can choose $g_1 \in G$ such that $g_1 \cdot x_1 \neq x_1$. Furthermore, since $\chi |_{H_{x_1}} \neq 0$, 
  we can choose a generator $h \in Y_1$ such that $\chi (h) \neq 0$. We may assume without loss of generality that $\chi(h)>0$. Define 
  $t := {^{g_1}h} \in H_{g_1 \cdot x_1}$. We take $\mathcal{X} = Y_1 \cup \ldots \cup Y_n \cup Z \cup \{t\}$ as a generating set for $\Gamma$ and we show
  that the conditions of Theorem \ref{Renz1} are satisfied.
  
  If $y \in (Y_1\cup \ldots \cup Y_n)^{\pm 1}$ then $t$ and $y$ commute in $\Gamma$, hence $w_y := y$ is word that represents $t^{-1} y t$. Also,
  $v_{\chi}(w_y) = \chi(y)$ and 
  \[v_{\chi}(t^{-1} y t) \leq \chi(t^{-1}y) = \chi(y) - \chi(t) < \chi(y),\]
  so $v_{\chi}(t^{-1}yt) < v_{\chi}(w_y)$.
  
  If $z \in Z^{\pm 1}$, there are two cases: $z \in G_{g_1 \cdot x_1}$ or $z \notin G_{g_1 \cdot x_1}$. In the first case $z$ and $t$ commute in $\Gamma$,
  so we may proceed as in the previous paragraph: we take the word $w_z := z$, which represents $t^{-1}zt$ and satisfies $v_{\chi}(t^{-1}zt) < v_{\chi}(w_z)$.
  If $z \notin G_{g_1 \cdot x_1}$ notice that $^{z}t$ and $t^{-1}$ lie in different copies of $H$ in $\Gamma$, therefore they commute, so:
  \[ t^{-1} z t = t^{-1} ({^{z}t}) z = ({^{z}t}) t^{-1} z = z t z^{-1} t^{-1} z.\]
  In this case define $w_z := z t z^{-1} t^{-1} z$. Observe that $v_{\chi}(w_z) = min \{0, \chi(z)\}$. If this minimum is $0$ then $\chi(z) \geq 0$, 
  and so $v_{\chi}(t^{-1} z t) = -\chi(t) < 0$. Otherwise $v_{\chi}(w_z) = \chi(z) < 0$ and so
  $v_{\chi}(t^{-1} z t) \leq \chi(t^{-1}z) = \chi(z) - \chi(t) < \chi(z)$. In both cases $v_{\chi}(t^{-1} z t) < v_{\chi}(w_z)$. 
  
  Thus $[\chi] \in \Sigma^1(\Gamma)$ by Theorem \ref{Renz1}.  
 \end{proof}
 
 In order to complete the proof of Theorem \ref{thmSigma1}, we only need to consider the cases where the restriction of $\chi$ to the copies of $H$ is 
 non-zero only for copies associated to orbits that are composed by only one element, and this is done by use of the direct product formula, as follows.
 
 \begin{theorem} \label{reswreath2}
Let $\Gamma = H \wr_X G$ be a finitely generated wreath product and set $M = \oplus_{x \in X} H_x \subseteq \Gamma$. Let $\chi: \Gamma \to \mathbb{R}$
be a non-zero character such that $\chi |_M \neq 0$. 
Then $[\chi] \in \Sigma^1(\Gamma)$ if and only if at least one the following conditions holds:
\begin{enumerate}
 \item The set $T = \{x \in X | \hbox{ } \chi |_{H_x} \neq 0\}$ has at least two elements;
 \item $T = \{x_1\}$ and $\chi |_G \neq 0$;
 \item $T = \{x_1\}$ and $[\chi |_{H_{x_1}}] \in \Sigma^1(H)$.
\end{enumerate}
 \end{theorem}
 
 \begin{proof}
  By Proposition \ref{sigma1viarenz} it is enough to consider the case where $G \cdot x = \{x\}$ for all $x \in T$. Notice that in this case $T$ must be 
  finite, since each of its elements is an entire orbit of $G$ on $X$, and there are finitely many of those. Let $P = \prod_{x \in T} H_x$ and $X' = X \smallsetminus T$.
  Then
  \[ \Gamma = H \wr_X G \simeq P \times (H \wr_{X'} G).\]
  If $T$ has at least two elements, then $[\chi |_P] \in \Sigma^1(P)$ and hence $[\chi] \in \Sigma^1(\Gamma)$, by two applications of the direct product 
  formula for $\Sigma^1$. If $T = \{x_1\}$, the formula gives us exactly that $[\chi] \in \Sigma^1(\Gamma)$ if and only if one of conditions
  $(2)$ or $(3)$ holds, since $\chi |_G \neq 0$ if and only if $\chi |_{H \wr_{X'} G} \neq 0$.
  \end{proof}
\end{section}

\begin{section}{Graph-wreath products} \label{gwp}
We now digress a bit and obtain a generalization of the results of Section \ref{sw} to a wider class of groups. Besides being interesting in its own
right, this will be useful in the analysis of the 
$\Sigma^2$-invariants of wreath products.

Given two groups $G$ and $H$ and $K$ a $G$-graph, the \textit{graph-wreath product} $H \grw_K G$ is defined by Kropholler and Martino \cite{KrophollerMartino} 
as the semi-direct product $H^{\langle K \rangle} \rtimes G$, where $H^{\langle K \rangle}$ is the graph product of $H$ with respect to the graph $K$ (that 
is, $H$ is the group associated to every vertex of $K$). The action of $G$ is given by permutation
of the copies of $H$ according to the $G$-action on the vertex set of $K$.
When $K$ is the complete graph $H \grw_K G$ is simply $H \wr_X G$, where $X$ is the vertex set of $K$.
        
Kropholler and Martino showed that $H \grw_K G$ is finitely generated if and only if $G$ and $H$ are finitely generated 
and $G$ acts with finitely many orbits of vertices on $K$, that is, $H \grw_K G$ is finitely generated under the same conditions as $H \wr_X G$ is, where
$X$ is the vertex set of $K$.
 
 In what follows we fix $\Gamma = H \grw_K G$ and $M = H^{\langle K \rangle} \subseteq \Gamma$. We assume that $\Gamma$ is finitely generated and we decompose
 $X$ in orbits as $X = G \cdot x_1 \sqcup \ldots \sqcup G \cdot x_n$.
 Moreover, we choose finite generating sets $Z$ for $G$ and $Y_i$ for $H_{x_i}$ for all $i=1, \ldots, n$ and we denote 
 $\mathcal{X} = (\cup_{i=1}^n Y_i) \cup Z$, which is seen as a generating set for $\Gamma$.
 
     \begin{theorem} \label{p1}  
 Let $\chi: H \grw_K G \to \mathbb{R}$ be a non-zero character such that $\chi |_M = 0$.
 Then $[\chi] \in \Sigma^1(H \grw_K G)$ if and only if $[\chi |_G] \in \Sigma^1(G)$ and $\chi |_{G_x} \neq 0$ for all $x \in X$. 
     \end{theorem}

     \begin{proof}
 Let $N_K$ be the kernel of the obvious homomorphism $M \twoheadrightarrow \oplus_{x \in X} H_x$. Note that
  $N_K \subseteq ker(\chi)$ and that $\bar{\Gamma}:=\Gamma/N_K \simeq H \wr_X G$. It follows that $\chi$ induces a character 
 $\bar{\chi}: \bar{\Gamma} \to \mathbb{R}$. For an element $\gamma \in \Gamma$, we denote by $\bar{\gamma}$ its image in $\bar{\Gamma}$.
 
 If $[\chi] \in \Sigma^1(\Gamma)$, then $[\bar{\chi}] \in \Sigma^1(\bar{\Gamma})$ (again by Proposition A4.5 in
 \cite{Strebel}). Thus
 $[\chi |_G] \in \Sigma^1(G)$ and $\chi |_{G_x} \neq 0$ for all $x \in X$ by Theorem \ref{thmSigma1}.
 
Conversely, suppose that $[\chi |_G] \in \Sigma^1(G)$ and that $\chi |_{G_x} \neq 0$ for all $x \in X$. Then
 $[\bar{\chi}] \in \Sigma^1(\bar{\Gamma})$. We will show that this implies that $Cay(\Gamma;\mathcal{X})_{\chi}$ is connected. 
 
 We need to show that for all $\gamma \in \Gamma_{\chi}$, there is a path in $Cay(\Gamma;\mathcal{X})_{\chi}$ connecting $1$ and $\gamma$. 
 Given such a $\gamma$, notice that $\bar{\gamma} \in \bar{\Gamma}_{\bar{\chi}}$, so there must be a path from $1$ to
 $\bar{\gamma}$ in $Cay(\bar{\Gamma}; \bar{\mathcal{X}})_{\bar{\chi}}$. Its obvious lift to $Cay(\Gamma;\mathcal{X})$ with $1$ as initial vertex
 is a path in $Cay(\Gamma;\mathcal{X})_{\chi}$ that ends at an element of the form $\gamma n$, with $n \in N_K$. If we can connect $\gamma$
 to $\gamma n$ inside $Cay(\Gamma;\mathcal{X})_{\chi}$ we are done. For that it suffices to find a path in $Cay(\Gamma;\mathcal{X})_{\chi}$ connecting $1$ and $n$,
 and then act with $\gamma$ on the left.
 
 Since $N_K \subseteq M$, each $n \in N_K$ can be written as:
 \begin{equation} \label{prim}
  n = (^{g_1}h_1) (^{g_2}h_2) \cdots (^{g_k}h_k),
 \end{equation}
 with $h_j \in \cup_{i=1}^n Y_i^{\pm 1}$ and $g_j \in G$ for all $j$. Even more, we may assume that each $\chi(g_j) \geq 0$. Indeed, since 
 $\chi |_{G_x} \neq 0$ for all $x$, we can always pick $t_j \in G$ such that $\chi(t_j) > 0$ and $^{t_j} h_j = h_j$. Then we may change $g_j$ for $g_j t_j^{k_j}$
 in \eqref{prim}, where $k_j$ is some integer such that $k_j \chi(t_j) \geq - \chi(g_j)$.
 
 But if $\chi(g_j) \geq 0$ then $g_j \in G_{\chi |_{G}}$, and since $[\chi |_{G}] \in \Sigma^1(G)$, we can choose words $w_j$ in $Z^{\pm 1}$ 
 representing $g_j$ and such that $v_{\chi}(w_j) \geq 0$. Finally, the word
 \[ w = (w_1 h_1 w_1^{-1})(w_2 h_2 w_2^{-1}) \cdots (w_k h_k w_k^{-1})\]
 is the label for a path connecting $1$ and $n$ in $Cay(\Gamma; \mathcal{X})_{\chi}$, by the choice of each $w_j$ together with the fact 
 that $\chi(h_j) = 0$ for all $j$ by hypothesis. 
     \end{proof}     
     
The above result will be needed only in a special case, namely when $K$ is a graph without edges, so that $\Gamma \simeq (\ast_{x \in X} H_x) \rtimes G$.
\end{section}

\begin{section}{The \texorpdfstring{$\Sigma^2$}{Sigma2}-invariant}
  Renz's paper \cite{Renz} also brings a criterion for the invariant $\Sigma^2$. In order to state it, we need to introduce the concept of a 
  diagram over a group presentation, for which we follow \cite{Bridson}. Fix an orientation on $\mathbb{R}^2$. Define a \textit{diagram} to be a subset 
  $M \subseteq \mathbb{R}^2$ endowed with the structure of a finite combinatorial $2$-complex. Thus to each $1$-cell of $M$ correspond two opposite 
  directed edges. If $\langle \mathcal{X} | \mathcal{R} \rangle$ is a presentation for a group $\Gamma$, a \textit{labeled diagram over} $\langle \mathcal{X} | \mathcal{R} \rangle$ is a diagram $M$ endowed with an edge labeling 
  satisfying:
  \begin{enumerate}
   \item The edges of $M$ are labeled by elements of $\mathcal{X}^{\pm 1}$;
   \item If an edge $e$ has label $x$, then its opposite edge has label $x^{-1}$;
   \item The boundary of each face of $M$, read as a word in $\mathcal{X}^{\pm 1}$, beginning at any vertex and proceeding with either orientation, 
         is either a cyclic permutation of some $r \in \mathcal{R}^{\pm 1}$, or a word of the form $tt^{-1}t^{-1}t$ for some $t \in  \mathcal{X}^{\pm 1}$.
  \end{enumerate}
A labeled diagram $M$ is said to be \textit{simple} if it is connected and simply connected.

\begin{rem}
This is a weakening of the definition of the usual \textit{van Kampen diagrams}. In fact, a simple diagram $M$, with a vertex chosen as a base point, 
differs from a van Kampen diagram only by the fact that it can have what we call \textit{trivial faces}, that is, those labeled by $tt^{-1}t^{-1}t$
for some $t \in  \mathcal{X}^{\pm 1}$. This weakening has the effect of simplifying the drawing of some diagrams that we will consider in the 
sequence (see \cite{Renz}, Subsection 3.3).
\end{rem}
 
% How to avoid the problem: add an extra generator $s$ and two extra relations $ts^{-1}$ and $[t,s]$. Then divide faces labeld by $tt^{-1}t^{-1}t$
% into three faces, two with label $ts^{-1}$, and one $[t,s]$. 

Suppose that we are given a simple diagram $M$ with a base point $u$ (a vertex in the boundary of $M$) and an element $\gamma \in \Gamma$. Then to each vertex 
$u'$ of $M$ corresponds a unique element of $\Gamma$, given by $\gamma \eta$, where $\eta$ is the image in $\Gamma$ of the label of any path connecting $u$ and
$u'$ inside $M$. In particular, the given group element $\gamma$ corresponds to the base point $u$. For any character 
$\chi: \Gamma \to \mathbb{R}$ we define the $\chi$-\textit{valuation} of $M$ with respect to $u$ and $\gamma$, denoted by $v_{\chi}(M)$, to be the 
minimum value of $\chi(g)$ when $g$ runs over the elements of $\Gamma$ corresponding to the vertices of $M$. 

Now, suppose that $\Gamma$ is finitely presented (with $\langle \mathcal{X} | \mathcal{R} \rangle$ a finite presentation) and assume $[\chi] \in \Sigma^1(\Gamma)$.
Then we can
distinguish an element $t \in \mathcal{X}^{\pm 1}$ with $\chi(t) > 0$ with which we can apply Renz's criterion for $\Sigma^1$: for each 
$x \in \mathcal{X}^{\pm 1} \smallsetminus \{t, t^{-1}\}$ we can associate a word $w_x$ in $\mathcal{X}^{\pm 1}$ that represents $t^{-1}xt$ and for which
$v_{\chi}(t^{-1} x t) < v_{\chi}(w_x)$. Additionally, we put $w_t := t$ and 
$w_{t^{-1}} := t^{-1}$. If $r = x_1 \cdots x_n \in \mathcal{R}^{\pm 1}$, we define:
\[ \hat{r} := w_{x_1} \cdots w_{x_n}.\]
We are now ready to enunciate the criterion for $\Sigma^2$.

\begin{theorem}[\cite{Renz}, Theorem 3] \label{Renz2}
 Let $\Gamma$, $\mathcal{X}$ and $t$ be as above. Suppose that the set $\mathcal{R}$ of defining relations contains some cyclic permutation of the words $t^{-1} x t w_x^{-1}$, for all $x \in \mathcal{X}^{\pm 1}$.
Then  $[\chi] \in \Sigma^2(\Gamma)$ if and only if for each $r \in \mathcal{R}^{\pm 1}$ there exist a simple diagram $M_{\hat{r}}$ and vertex $u$
in its boundary, such that both the following conditions hold:
\begin{enumerate}
 \item The boundary path of $M_{\hat{r}}$, read from $u$, has as label the word $\hat{r}$;
 \item $v_{\chi}(r) < v_{\chi}(M_{\hat{r}})$, where the valuation of $M_{\hat{r}}$ is taken
with respect to the base point $u$ and the element $t \in \Gamma$. 
\end{enumerate}
\end{theorem}

Now, recall that a wreath product $H \wr_X G$ is finitely presented if and only if $G$ and $H$ are finitely presented, $G$ acts diagonally on $X^2$ with finitely
many orbits and the stabilizers of the $G$-action on $X$ are finitely generated. This is the result by Cornulier \cite{Cornulier}.

We will apply Theorem \ref{Renz2} to show that if $\Gamma = H \wr_X G$ is finitely presented and if $\chi: \Gamma \to \mathbb{R}$ 
is a character such that $\chi |_{H_{x_1}} \neq 0$ for some $x_1 \in X$ with $|G \cdot x_1| = \infty$, then $[\chi] \in \Sigma^2(\Gamma)$.

We start by assuming that $G$ acts transitively on $X$, with $X = G \cdot x_1$. Let $\langle Y | R \rangle$ and $\langle Z | S \rangle$ be finite presentations
for $H$ and $G$, respectively. We may assume that $Z$ contains a generating set $E$ for the stabilizer subgroup $G_{x_1}$ and a set $J$ of representatives
for the non-trivial double cosets of $(G_{x_1},G_{x_1})$ in $G$, since both $E$ and $J$ can be taken to be finite by the proof of the main theorem in \cite{Cornulier}.

We think of $\Gamma = H \wr_X G$ with the presentation considered by Cornulier. So $\Gamma$ is generated by the set $Y \cup Z$, subject 
to the following defining relations:
 \begin{enumerate}
  \item[(1)] $r$, for all $r \in R$ (defining relations for $H$);
  \item[(2)] $s$, for all $s \in S$ (defining relations for $G$);
  \item[(3)] $[^{g}y_1, y_2]$, for $g \in J$, $y_1,y_2 \in Y$;
  \item[(4)] $[e, y]$, for $e \in E$ and $y \in Y$.
  \end{enumerate} 
Let us adapt a bit this presentation. We are under the hypothesis that $\chi |_{H_{x_1}} \neq 0$ and $|G \cdot x_1 | = \infty$. 
We may assume without loss of generality that $\chi(h) > 0$ for some $h \in Y$. Choose $g_i \in Z$, for $1 \leq i \leq 5$, such that 
$\{x_1\} \cup \{g_i \cdot x_1 \hbox{ }|\hbox{ } 1 \leq i \leq 5\}$ is a set with exactly six elements (of course we may assume that $Z$ contains elements
$g_i$ with this property). Define
\[t_i := {^{g_i}h},\]
for $i=1, \ldots, 5$. 
Then $\Gamma$ is generated by $Y \cup Z \cup \{t_i \hbox{ }|\hbox{ } 1 \leq i \leq 5\}$, subject to the following defining relations:
 \begin{enumerate}
  \item[(1)] $r$, for all $r \in R$ (defining relations for $H$);
  \item[(2)] $s$, for all $s \in S$ (defining relations for $G$);
  \item[(3)] $[^{g}y_1, {^{g'}}y_2]$, for all $y_1,y_2 \in Y \cup \{t_i \hbox{ }|\hbox{ } 1 \leq i \leq 5\}$ and $g,g' \in Z \cup \{1\}$ whenever the commutator
             $[^{g}y_1, {^{g'}}y_2]$ is indeed a relation in $\Gamma$;
  \item[(4)] $[e, y]$, for $e \in E$ and $y \in Y$ and  \\
             $[z,t_1]$, for $z \in Z \cap G_{g_1 \cdot x_1}$;
  \item[(5)] $g_i h g_i^{-1} t_i^{-1}$, for $1 \leq i \leq 5$.
 \end{enumerate}

\begin{rem}
We could write the conditions of item $(3)$ in a more precise way, but it would require writing many cases. If $y_1 \in Y$ and $y_2 = t_1$, for example, then
$[^{g}y_1, {^{g'}}y_2]$ is a defining relation if $g \cdot x_1 \neq (g'g_1) \cdot x_1$.
\end{rem}
 
Note that we have added a few relations of the types (3) and (4), but clearly they are consequences of the others. Furthermore, the set of relations is clearly
still finite. 

Set $t = t_1$. We will continue using the notation of Proposition \ref{sigma1viarenz}. Thus for $y \in Y^{\pm 1}$ we have chosen 
$w_y = y$. If $z \in Z^{\pm 1}$, then $w_z = z$ if $z \in G_{g_1 \cdot x_1}$ and $w_z = ztz^{-1}t^{-1}z$ otherwise. Moreover, since $t_i$ and $t$ 
commute in $\Gamma$ for all $1 \leq i \leq 5$, we can define $w_{t_i}: = t_i$ and $w_{t_i^{-1}}: = t_i^{-1}$. 

Let us check that the chosen presentation satisfies the conditions of Theorem \ref{Renz2}. First, the set of defining relations contains the relations 
$t^{-1} x t w_x^{-1}$. Indeed if $y \in Y^{\pm 1} \cup \{t_i \hbox{ }|\hbox{ } 1 < i \leq 5\}^{\pm 1}$ then $t^{-1} y t w_y^{-1}$ is a relation
of type $(3)$, since $w_y = y$. If $z \in Z^{\pm 1} \cap G_{g_1 \cdot x_1}$, then $w_z = z$ and $t^{-1} z t z^{-1}$ is a relation of type (4). 
Finally, if $z \in Z^{\pm 1}$ but $z \notin G_{g_1 \cdot x_1}$, then $w_z = ztz^{-1}t^{-1}z$ and
\[t^{-1}zt w_z^{-1} = t^{-1}zt z^{-1} t z t^{-1} z^{-1} = t^{-1}(^{z}t)t(^{z}t)^{-1},\]
which is a cyclic permutation of $[^{z}t,t]$, a relation of type (3).

 According to Theorem \ref{Renz2}, now we need to apply the transformation $r \mapsto \hat{r}$ to each defining relation and then find a simple diagram 
 $M_{\hat{r}}$ satisfying the stated conditions. The following subsections are devoted to the verification of the existence of these diagrams.
 Observe that we do not need to consider the inverses of defining relations, since any simple diagram for $\hat{r}$ 
is a simple diagram for the inverse of $\hat{r}$ if we read its boundary backwards.

 \begin{subsection}{Relations of type (1)}

 Note that the relations of type $(1)$ involve only generators in $Y^{\pm 1}$. But $w_y = y$ 
 for all $y \in Y^{\pm 1}$, so $\hat{r} = r$ whenever $r$ is a relation of type $(1)$. Thus the one-faced diagram $M$ that represents the
 relation $r$, with base point corresponding to $t$, is already a choice for $M_{\hat{r}}$, since its $\chi$-value is increased by $\chi(t) > 0$.
 
 In Figure \ref{fig1} we represent the diagram $M_{\hat{r}}$, for $r= \hat{r} =  y_1 y_2 y_3 y_4$, as the internal square of the figure. The 
 external boundary represents the path beginning at the base point $1 \in \Gamma$ and with label the original relation $r$. The edges labeled by
 $t$ indicate that $\hat{r}$ is obtained from $r$ by conjugation by $t$.
 
 \begin{figure}
\[ \xymatrix{ 1 \ar[rrr]^{y_1}  \ar[dr]^t &                  &          & \bullet \ar[ddd]^{y_2} \ar[ld]_t \\
                                         & \bullet \ar[r]^{y_1} & \bullet \ar[d]^{y_2} &                   \\
                                         & \bullet \ar[u]^{y_4}        & \bullet  \ar[l]^{y_3}      &                   \\
             \bullet \ar[uuu]^{y_4} \ar[ur]^t         &                  &          & \bullet \ar[lu]_t \ar[lll]^{y_3}          }
             \]
\caption{Diagram for a relation of type (1), $r=\hat{r} = y_1 y_2 y_3 y_4$}
\label{fig1}             
\end{figure}
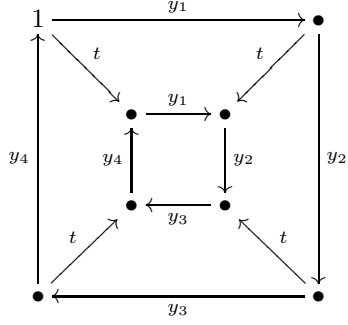

 \end{subsection}
 
 \begin{subsection}{Relations of type (2)}
 Since $\chi(h) \neq 0$, the order of $h$ in $\Gamma$ is infinite. Consider the subgroup
  \[\Gamma_0 := \langle h, G \rangle \leqslant \Gamma.\]
 Notice $\Gamma_0 \simeq \mathbb{Z} \wr_X G$. Let $\chi_0$ be the restriction of $\chi$ to $\Gamma_0$. The group $\Gamma_0$ is an extension of an abelian 
 group $A = \oplus_{x \in X} \mathbb{Z}$ by $G$, so it follows from Theorem \ref{ThmKochloukova}
 that $[\chi_0] \in \Sigma^2(\Gamma)$, as $\chi_0 |_A \neq 0$. Now choose a presentation for $\Gamma_0$ that is
 compatible with the chosen presentation for $\Gamma$:
 write the same presentation with $Y = \{h\}$ and discard the relations of type $(1)$. 
 Naturally, this presentation satisfies the hypothesis of Theorem \ref{Renz2}.
 
 Let $r = z_1 \cdots z_n$ be a relation of type $(2)$. We can see $r$ as a relation in $\Gamma_0$. By Theorem \ref{Renz2} there is a simple diagram 
 $M_{\hat{r}}$, with respect to the presentation of $\Gamma_0$, whose base point corresponds to $t$ and such that $\partial M_{\hat{r}} = \hat{r}$ and $v_{\chi}(r) < v_{\chi}(M_{\hat{r}})$.
 But all the relations in the chosen presentation of $\Gamma_0$ are also relations in the original presentation of $\Gamma$, after identifying the generating sets.
 Then $M_{\hat{r}}$, if seen as a diagram over the presentation of $\Gamma$, is the diagram we wanted. 
 \end{subsection}

 \begin{subsection}{Relations of the types (4) or (5)}
 All cases are similar: we can obtain simple diagrams whose only vertices are those of the boundary, that is, those that are defined by the word $\hat{r}$.
 In this case, the diagram automatically satisfies the hypothesis about its $\chi$-value, exactly as in the case of the relations of type $(1)$. 
 See the diagram for the relation $g_1 h g_1^{-1} t^{-1}$ in Figure \ref{fig2}. As before, the external boundary represents
 $r = g_1 h g_1^{-1} t^{-1}$, and the diagram $M_{\hat{r}}$ itself is the internal diagram composed by the five squares. Again, the edges labeled
 by $t$ and with origin at some point in the external path indicate conjugation by $t$ and represent the growth of $v_{\chi}$ from $r$ to $\hat{r}$.
 
 \begin{figure}
\[
 \xymatrix{ 1 \ar@{>}[rrrrrrr]^{g_1} \ar[ddd]_t \ar[dr]^t &  &  &  &   &  &   & \bullet \ar[ddd]^h \ar[ld]_t \\
 & \bullet \ar[r]^{g_1} \ar[d]_{t} & \bullet \ar[r]^{t} \ar[d]_h & \bullet  \ar[d]_h & \bullet \ar[l]_{g_1}  \ar[d]^t & \bullet  \ar[l]_{t} \ar[r]^{g_1} \ar[d]^t & \bullet \ar[d]^h  &  \\
 & \bullet \ar[r]_{g_1}  & \bullet \ar[r]_{t} & \bullet & \bullet \ar[l]^{g_1}  & \bullet  \ar[l]^{t} \ar[r]_{g_1}  & \bullet  &  \\
   \bullet \ar[ur]_t \ar[rrrrrrr]_{g_1} &  &  &  &   &  &   & \bullet  \ar[ul]^t 
   }
\]
\caption{Diagram for relations of type (5)}
\label{fig2}
\end{figure}
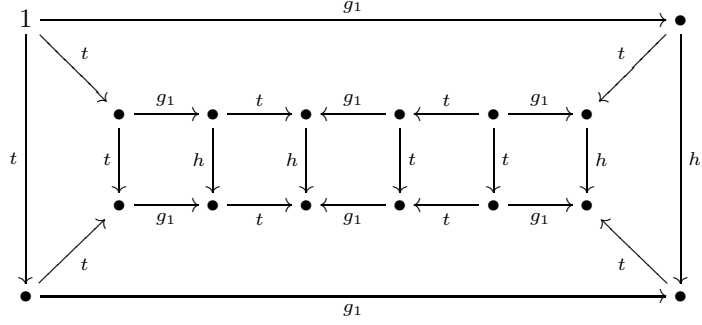

 Figure \ref{fig2} is an illustration of the case when $g=g_1 \notin G_{g_1 \cdot x}$, when the word $w_g$ is more complicated. If the letter representing
 an element of $G$ is an element of $G_{g_1 \cdot x}$, the argument is simpler: all the letters involved in the 
 relation commute with $t$, so $r = \hat{r}$ and the argument follows as in the case of the relations of type $(1)$.
  \end{subsection}
 
 \begin{subsection}{Relations of type (3)}
 Let $y \in Y \cup \{t_i \hbox{ }| \hbox{ } 1 \leq i \leq 5\}$ and $g \in Z$. Let $\eta_{g,y}$ be the word obtained from $g y g^{-1}$ by applying the transformation
 that takes each letter $\alpha$ to $w_{\alpha}$:
 \begin{equation} \label{decomp1}  
 \eta_{g,y} = (^{g}t) t^{-1} (^{g}y) t (^{g}t)^{-1},
 \end{equation}
 if $g \notin G_{g_1 \cdot x_1}$, or
 \begin{equation}  \label{decomp2}
 \eta_{g,y} = {^{g}y}, 
 \end{equation}
 if $g \in G_{g_1 \cdot x_1}$. If $g=1$, put $\eta_{1,y} := y$. In all cases we see that $\eta_{g,y}$ is a product of subwords representing elements of at most $3$ copies of $H$ in $\Gamma$. 
 Indeed, $^{g}t$ and $(^{g}t)^{-1}$ are elements of $H_{gg_1 \cdot x_1}$, while $t$ and $t^{-1}$ are elements of $H_{g_1 \cdot x_1}$ and, finally,
 $^{g}y$ is an element of $H_{g \cdot x_1}$ or $H_{gg_i \cdot x_1}$ for some $1 \leq i \leq 5$,
 depending on $y$.  
 
 Consider $r = [^{g}y_1, {^{g'}y_2}]$, a relation of type $(3)$. The word $\hat{r} = [\eta_{g,y_1}, \eta_{g',y_2}]$ is a relation in $\Gamma$,
 so we can always find a simple diagram $M_1$ with some base point corresponding to $t$ and such that $\partial M_1 =  \hat{r}$. If $v_{\chi}(M_1) > v_{\chi}(r)$ we are done.
 Otherwise there is a vertex $p$ of $M_1$ such that $\chi(p) \leq v_{\chi}(r)$. Notice that this vertex can not lie on the boundary of $M_1$, since
 $v_{\chi}(\hat{r}) > v_{\chi}(t^{-1}rt)$.
 
 Now, the commutator between the words $\eta_{g,y_1}$ and $\eta_{g',y_2}$ is a product of elements of the form
 $^{z}y$, with $z \in Z \cup \{1\}$ and $y \in Y^{\pm 1} \cup \{t_1, \ldots, t_5\}^{\pm 1}$. By the remarks above, these elements lie in at most
 most five different copies of $H$, one of them being indexed by $g_1 \cdot x_1$ when five copies do pop up. It follows that for some $u \in \{h,t_2, t_3,t_4, t_5\}$,
 the words $[^{z}y, u]$ are defining relations for all the subwords $^{z}y$ appearing in $\hat{r} = [\eta_{g,y_1}, \eta_{g',y_2}]$ (we consider the subwords
 $^{z}y$ that appear when $\eta_{g,y_1}$ and 
 $\eta_{g',y_2}$ are written exactly as in \eqref{decomp1} or \eqref{decomp2}).
 Observe that $\chi(u) = \chi(h)>0$ in all cases. So we can build a diagram $M_2$ by surrounding 
 $M_1$ with faces representing the commutators $[^{z}y, u]$, for all these subwords $^{z}y$.

 Clearly the boundary of $M_2$ is also labeled by $\hat{r}$. If we set as base point the vertex on the new boundary corresponding to the base point of $M_1$ (that is, the 
 one joined to it by an edge with label $u$), 
 then the $\chi$-value of the interior points (including $p$) is increased by $\chi(u)>0$, so that $v_{\chi}(M_2) > v_{\chi}(M_1)$. Repeating 
 this process finitely many times, we obtain a simple diagram $M_n$ satisfying the conditions of the theorem.
  \end{subsection} 
  
 We record what we have proved in the following proposition.
 \begin{prop} \label{propsigma2}
  Let $\Gamma = H \wr_X G$ be a finitely presented wreath product and let $M = \oplus_{x \in X} H_x \subseteq \Gamma$. Suppose that $G$ acts transitively on the 
  infinite set $X$. If $\chi: \Gamma \to \mathbb{R}$ is a character with $\chi |_M \neq 0$, then $[\chi] \in \Sigma^2(\Gamma)$.
 \end{prop}
 
The arguments above essentially contain what we need when $G \cdot x$ is infinite for some $x \in X$ such that $\chi |_{H_x} \neq 0$ (but the 
$G$-action on $X$ is not necessarily transitive), so we will only indicate in the proof of the following theorem how to deal with this case. 

Recall that we denote by $T$ the set of elements $x \in X$ such that $\chi |_{H_x} \neq 0$. Notice that if $T= \{x_1\}$, 
then $\Gamma$ is a direct product
\[\Gamma \simeq H_{x_1} \times (H \wr_{X'} G),\]
where $X' = X \smallsetminus \{x_1\}$. Then the direct product formula and the results on the $\Sigma^1$-invariants of wreath products already contain all the information 
we need. The remaining cases are all part of the following theorem, which includes Theorem \ref{thmB}.
 
\begin{theorem} \label{sigma2viarenz}
   Let $\Gamma = H \wr_X G$ be a finitely presented wreath product and let $M = \oplus_{x \in X} H_x \subseteq \Gamma$.
   Suppose that the set
   \[ T = \{x \in X \hbox{ }| \hbox{ } \chi |_{H_x} \neq 0\}. \]
  has at least two elements. Then $[\chi] \in \Sigma^2(\Gamma)$ if and only if at least one of the following conditions holds:
   \begin{enumerate}
    \item[(1)] $[\chi |_{H_x}] \in \Sigma^1(H)$ for some $x \in T$;
    \item[(2)] $\chi|_G \neq 0$;
    \item[(3)] $T$ has at least three elements.
   \end{enumerate}
  \end{theorem}

  \begin{proof}
   Suppose first that $T$ is a finite set and consider the subgroup $B = \cap_{x \in T} G_x \leqslant G$. It is of finite index in $G$, so $\Gamma_1 = H \wr_X B$ is of finite index in 
   $\Gamma$. Notice that \[\Gamma_1 \simeq (\prod_{x \in T} H_x) \times (H \wr_{X'} B).\] 
   Denote $P =\prod_{x \in T} H_x$ and $Q = H \wr_{X'} B$. The fact that $T$ has at least two elements implies that $[\chi |_P] \in \Sigma^1(P)$,
   by the direct product formula. By applying the formula again, now to the product $\Gamma_1 = P \times Q$, we get that 
   $[\chi |_{\Gamma_1}] \in \Sigma^2(\Gamma_1)$ if and only if $[\chi |_P] \in \Sigma^2(P)$ or $\chi |_Q \neq 0$. The former happens if and only
   if at least one of conditions (1) or (3) is satisfied (once again, by the direct product formula), while the latter clearly happens if and only if $\chi |_B \neq 0$, which in turn is equivalent with $\chi |_G \neq 0$, since
   $B$ is a subgroup of finite index. Finally, since the index of $\Gamma_1$ in $\Gamma$ is finite, we are done by Theorem \ref{finiteindexthm}. 
   
   We are left with the case where $T$ is infinite and we want to show that $[\chi] \in \Sigma^2(\Gamma)$. Since $G$ acts on $X$ with finitely many orbits, 
   there must be some $x_1 \in T$ such that $| G \cdot x_1 | = \infty$. We adapt the proof of Proposition \ref{propsigma2} putting the orbit of $x_1$
   in a distinguished position.
   
   Choose $x_2$, $\ldots$, $x_n \in X$ such that $X = \bigsqcup_{j=1}^n G \cdot x_j$. For each $j$ choose a finite generating set $E_j$ for the stabilizer subgroup
   $G_{x_j}$. For each pair par $(i,j)$, with $ 1 \leq i,j \leq n$, choose a finite set $J_{i,j}$ of representatives of the non-trivial double cosets of 
   $(G_{x_i}, G_{x_j})$ in $G$. Finally, choose finite presentations $\langle Y | R \rangle$ and 
   $\langle Z | S \rangle$ for $H$ and $G$ respectively. We may assume that $Z$ contains $E_j$ and $J_{i,j}$ for all $1 \leq i,j \leq n$.
   
   A finite presentation for $\Gamma$, adapted from the presentation given by Cornulier \cite{Cornulier}, can be given as follows. For each $1 \leq i \leq n$ 
   we associate a copy $\langle Y_i | R_i \rangle$ of the presentation for $H$ and, as before, we define $t_i := {^{g_i}h}$ for some $g_i \in Z$ and 
   $h \in Y_1$ with $\chi(h)> 0$ and $|\{x_1\} \cup \{g_i \cdot x_1 | 1 \leq i \leq 5\}| = 6$.
   We think of $\Gamma$ as generated by $(\bigcup_{i=1}^n Y_i) \cup Z \cup \{t_i \hbox{ } | \hbox{ } 1 \leq i \leq 5\}$ and subject to the 
   defining relations given by:
   
   \begin{enumerate}
  \item[(1)] $r$, for all $r \in \bigcup_{i=1}^n R_i$ (defining relations for the copies of $H$);
  \item[(2)] $s$, for all $s \in S$ (defining relations for $G$);
  \item[(3)] $[^{g}y_1, {{^{g'}}y_2}]$, for $y_1,y_2 \in (\bigcup_{i=1}^n Y_i) \cup \{t_i \hbox{ }|\hbox{ } 1 \leq i \leq 5\}$ and 
             $g,g' \in Z \cup \{1\}$ whenever $[^{g}y_1, {^{g'}}y_2]$ is indeed a relation in $\Gamma$;
  \item[(4)] $[e_i, y_i]$, for all $e_i \in E_i$, $y_i \in Y_i$ and $1 \leq i \leq n$ and \\ $[z,t_1],$ for all $z \in Z \cap G_{g_1 \cdot x_1}$;
  \item[(5)] $g_i h g_i^{-1} t_i^{-1}$, for $1 \leq i \leq 5$.
 \end{enumerate}

 Set $t=t_1$. We use again the notation of Proposition \ref{sigma1viarenz}. So we use the same words $w_z$ if $z \in Z^{\pm 1}$, and
 $w_y = y$ for all other generators $y$. It is clear that 
 the set of defining relations above still satisfies the hypothesis of Theorem \ref{Renz2}. The construction of the diagrams associated
 to each defining relation can be done exactly as in the case where the action is transitive, as we will argue below. The key fact 
 is that the generators coming from copies of $H$ associated to all other orbits of $G$ (other than $G \cdot x_1$) commute
 with $t=t_1$.
 
 First, notice that the construction of the diagrams associated to the relations of types (1) or (4) in the case of a transitive action depends only on the fact 
 that $[t,y]$ is a defining relation for all $y \in Y=Y_1$. 
 But $t_1$ commutes also with all elements of $Y_2 \cup \ldots \cup Y_n$, so the construction can be carried out in the same way. For the case
 of relations of type (3), it was only necessary that for any generators $g, g' \in Z \cup \{1\}$ and $y,y' \in Y_1$, we could find some 
 $u \in \{h,t_2, \ldots, t_5\}$ that commutes with all the following elements: $t$, $^{g}t$, $^{g'}t$, $^{g}y$ and $^{g'}y'$. If we allow $y$
 to be an element of $Y_2 \cup \ldots \cup Y_n$, then any $u$ that commutes with $t$, $^{g}t$, $^{g'}t$ and $^{g'}y'$ will do it, since $^{g}y$ commutes 
 any choice of $u$. Thus the five options for $u$, coming from different copies of $H$, are enough to let us repeat the argument. Similar considerations cover
 the cases where either only $y'$, or both $y$ and $y'$ are elements of $Y_2 \cup \ldots \cup Y_n$.
 
 This is all we needed to check, since relations of types (2) and (5) do not involve any of the new generators. 
  \end{proof}
  \end{section}

  \begin{section}{Some observations about \texorpdfstring{$\Sigma^2$}{Sigma2}} \label{obs}
  Let $\Gamma$ be a finitely presented group and let $[\chi] \in S(\Gamma)$. Let $\langle \mathcal{X} | \mathcal{R} \rangle$ be a finite presentation for $\Gamma$.
  Denote by $C = Cay(\Gamma; \langle \mathcal{X} | \mathcal{R} \rangle)$ the associated Cayley complex and by $C_{\chi}$ the full subcomplex of $C$ spanned by $\Gamma_{\chi}$.
  The canonical action of $\Gamma$ on $C$ restricts to an action by the monoid $\Gamma_{\chi}$ on $C_{\chi}$. 
  
  \begin{rem}
   If a monoid $K$ acts on some set $X$ we still say that the sets $K \cdot x$ are \textit{orbits}. By ``$K$ has finitely many orbits on
   $X$'' we mean that there are elements $x_1, \ldots, x_n \in X$ such that $X = \bigcup_{j=1}^n K \cdot x_j$.
   \end{rem}
   
  The following lemma can be found in Renz's thesis \cite{RenzThesis}.
  
  \begin{lemma} \label{lemarenz}
$C_{\chi}$ has finitely many $\Gamma_{\chi}$-orbits of $k$-cells for $k \leq 2$.
  \end{lemma}
 
  \begin{proof}
 Denote by $D$ and $D_{\chi}$ the sets of $k$-cells of $C$ and $C_{\chi}$, respectively (for a fixed $k \leq 2$). We know that $\Gamma$ acts on $D$ 
 with finitely many orbits. Choose representatives $d_1, \ldots, d_n$ for these orbits so that $d_j \in D_{\chi}$ but
 $\gamma \cdot d_j \notin D_{\chi}$ for all $j$ and for all $\gamma \in \Gamma$ with $\chi(\gamma)< 0$. For this it suffices to take any representatives 
 $\tilde{d_1}, \ldots, \tilde{d_n}$ and then put $d_j := \gamma_j^{-1} \cdot \tilde{d_j}$, where $\gamma_j \in \Gamma$ is the vertex of $\tilde{d_j}$ with
 lowest $\chi$-value. Thus if $d \in D_{\chi}$,
 then $d = \gamma \cdot d_j$ for some $j$ and, by choice of $d_j$, we have that $\chi(\gamma) \geq 0$. So $D_{\chi} = \bigcup_{j=1}^n \Gamma_{\chi} \cdot d_j$.
  \end{proof}

 Denote by $F(\mathcal{X}, \chi)$ the submonoid of $F(\mathcal{X})$ consisting of the classes of reduced words $w$ with $v_{\chi}(w) \geq 0$. 
 Note that $F(\mathcal{X}, \chi)$ is indeed closed under the product, since $w_1, w_2 \in F(\mathcal{X}, \chi)$ implies $v_{\chi}(w_1 w_2) \geq 0$,
 and this property is preserved by elementary reductions (that is, canceling out terms of the form $x x^{-1}$ or $x^{-1}x$). 
 Let $R(\chi)$ be the subgroup of $F(\mathcal{X})$ consisting of the classes of reduced words $w$ that represent relations (that is 
 $w \in \langle \mathcal{R} \rangle^{F(\mathcal{X})}$) and such that $v_{\chi}(w) \geq 0$. Observe that $R(\chi) \subseteq F(\mathcal{X}, \chi)$ 
 and notice that $R(\chi)$ is indeed a subgroup, since $v_{\chi}(w) \geq 0$ implies $v_{\chi}(w^{-1}) \geq 0$ whenever $w$ is a relation.
 Finally, observe that $R(\chi)$ admits an action by the monoid $F(\mathcal{X}, \chi)$ via left conjugation.

 Now, let $r$ be a reduced word in $\mathcal{X}^{\pm 1}$ representing a relation in $\Gamma$, that is, 
 $r \in \langle \mathcal{R} \rangle^{F(\mathcal{X})}$. Suppose that $M$ is a van Kampen diagram over 
 $\langle \mathcal{X} | \mathcal{R} \rangle$ whose boundary, read in some orientation from some base point $p$, is exactly $r$. Then it holds in $F(\mathcal{X})$ that
  \begin{equation} \label{relacao}
r = {^{w_1}r_1} \cdots {^{w_n} r_n},  
 \end{equation}
 where each $r_i$ is a word read on the boundary of some face of $M$ and $w_i$ is the label for a path in $M$ connecting $p$ to a base point
 of the face associated to $r_i$. Both the facts that such a diagram exists and that $r$ can be written as above are consequences of 
 van Kampen's lemma (see Proposition 4.1.2 and Theorem 4.2.2 in \cite{Bridson}, for instance). 
 
 \begin{lemma} \label{lemaRfg}
 If $\chi: \Gamma \to \mathbb{R}$ is a character such that $C_{\chi} = Cay(G, \langle \mathcal{X} | \mathcal{R} \rangle)_{\chi}$ is $1$-connected, 
 then $R(\chi)$ is finitely generated over $F(\mathcal{X}, \chi)$.
 \end{lemma}
 
 \begin{rem}  
By ``$R(\chi)$ is finitely generated over $F(\mathcal{X}, \chi)$'' we mean that every element of $R(\chi)$ can be written as a product of elements of the form
 $^{w}s$, where $w \in F(\mathcal{X}, \chi)$ and $s \in S$ for some finite set $S \subseteq R(\chi)$.
 \end{rem}
 
 \begin{proof}
  Let $r \in R(\chi)$ and consider the path $\rho$ in $C$ beginning at $1$ and with label $r$. Notice that this path runs
  inside $C_{\chi}$, since $v_{\chi}(r) \geq 0$. Also, $\rho$ is clearly a loop and it must be nullhomotopic in $C_{\chi}$, since $C_{\chi}$ is $1$-connected.
  A homotopy from $\rho$ to the trivial path can then be realized by a van Kampen diagram $M$ with $v_{\chi}(M) \geq 0$ 
  (the valuation is taken with respect to $1$, seen both as base point in $C$ and group element). This is made precise by Theorem 2 in \cite{Renz}.
  
  Write $r$ as in \eqref{relacao}. Thus $r$ is a product of relations corresponding to the faces
  of $M$ conjugated on the left by elements of $F(\mathcal{X},\chi)$. Since $v_{\chi}(M) \geq 0$, such faces are faces of $C_{\chi}$, so by Lemma \ref{lemarenz} and using 
  that every element of $\Gamma_{\chi}$ can be written as a word in $F(\mathcal{X}, \chi)$, each
  $^{w_j}r_j$ can be rewritten as $^{u_j}s_j$ where ${u_j} \in F(\mathcal{X}, \chi)$ and each $s_j$ is a word read on the boundary of a face in a finite set
  $S$ of representatives of $\Gamma_{\chi}$-orbits of faces of $C_{\chi}$. It follows that $S$ is a finite generating set for $R(\chi)$ modulo the action of
  $F(\mathcal{X}, \chi)$.
  \end{proof}
 \end{section}

 \begin{section}{\texorpdfstring{$\Sigma^2$}{Sigma2} for characters with \texorpdfstring{$\chi  |_M = 0$}{chi M = 0}}
 We get back to a finitely presented wreath product $\Gamma = H \wr_X G = M \rtimes G$. 
 We consider now the non-zero characters 
 $\chi: \Gamma \to \mathbb{R}$ such that $\chi |_M = 0$. 
  
 In order to find sufficient conditions for $[\chi] \in \Sigma^2(\Gamma)$, we consider a nice action of $\Gamma$ on a complex. 
 We will briefly describe the construction in the proof of Theorem A in \cite{KrophollerMartino}, with the simplifications allowed by the fact
 that our situation is less general than what is considered in that paper.
 
 We are assuming that $\Gamma= H \wr_X G$ is finitely presented, so $H$ is also finitely presented. Choose a $K(H,1)$-complex $Y$, 
 with base point $\ast$, having a single $0$-cell and finitely many $1$-cells and $2$-cells. Let $Z = \oplus_{x \in X} Y_x$ be the 
 \textit{finitary product} of copies of $Y$ indexed by $X$, that is, $Z$ is the subset of the cartesian product $\prod_{x \in X} Y_x$ consisting
 on the families $(y_x)_{x \in X}$ such that $y_x$ is not the base point $\ast$ only for finitely many indices $x \in X$. It follows by the results
 in \cite{Davis} and \cite{DK} that $Z$ is an Eilenberg-MacLane space for $M = \oplus_{x \in X} H_x$. Notice that $Z$ has a natural cell structure.
 There is a single $0$-cell, given by the family $(y_x)_{x \in X}$ with $y_x = \ast$ for all $x$. For $n \geq 1$, an $n$-cell can be seen as a 
 product $c_1 \times \cdots \times c_k$ of cells of $Y$, 
 supported by some tuple $(x_1, \ldots, x_k) \in X^k$, such that $dim(c_1) + \ldots + dim(c_k) = n$.
 
 There is an obvious action of $G$ on $Z$. On the other hand, $M$ acts freely on the universal cover $E$ of $Z$. By putting together 
 these two actions, we get an action of $\Gamma = M \rtimes G$ on $E$. Clearly the $2$-skeleton of $E$ has finitely many $\Gamma$-orbits of 
 cells. Moreover, since the action of $M$ is free, the stabilizer subgroups are all conjugate to subgroups of $G$, and can be described as follows:
  \begin{enumerate}
  \item The stabilizer subgroup of any $0$-cell is a conjugate of $G$;
  \item For $n \geq 1$, the stabilizer subgroup of each $n$-cell contains a conjugate of the stabilizer $G_{(x_1, \ldots, x_n)}$ of some 
        $(x_1, \ldots, x_n) \in X^n$ as a finite index subgroup.
 \end{enumerate}
 We make stabilizers of $n$-cells correspond to stabilizers of $n$-tuples (rather than $k$-tuples, for $k \leq n$) by repeating some indices if 
 necessary. Also, the reason why we need to pass to a finite index subgroup is that cells of $Z$ written as products of cells of $Y$ may contain some
 repetition. For instance, a cell of $Z$ that arises as a product $c \times c$, supported by $(x_1,x_2)$, is also fixed by elements of $G$
 that interchange $x_1$ and $x_2$. This will also happen in the $\Gamma$-action on the universal cover $E$. 
 
 For groups admitting sufficiently nice actions on complexes, there is a criterion for the $\Sigma$-invariants.
 
 \begin{theorem}[\cite{Meinert}, Theorem B]
 Let $E'$ be a $2$-dimensional $1$-connected complex. Suppose that a group $\Gamma$ acts on $E'$ with finitely many orbits of cells. If 
 $\chi: \Gamma \to \mathbb{R}$ is a character such that $[\chi |_{\Gamma_c}] \in \Sigma^{2-dim(c)}(\Gamma_c)$ for all cells $c$ in $E'$, then 
 $[\chi] \in \Sigma^2(\Gamma)$.
 \end{theorem}

 We apply the theorem above with $E'$ being the $2$-skeleton of $E$. We obtain:
  
  \begin{prop} \label{suffcond}
  Suppose that $\Gamma = H \wr_X G$ is finitely presented and let $\chi: \Gamma \to \mathbb{R}$ be a non-zero character such that $\chi |_M = 0$.
  Suppose also that all properties below hold:
   \begin{enumerate}
    \item[(1)] $[\chi |_G] \in \Sigma^2(G)$;
    \item[(2)] $[\chi |_{G_x}] \in \Sigma^1(G_x)$ for all $x \in X$ and
    \item[(3)] $\chi |_{G_{(x,y)}} \neq 0$ for all $(x,y) \in X^2$.
      \end{enumerate}
  Then $[\chi] \in \Sigma^2(\Gamma)$.
  \end{prop}
  
 The fact that we can state the proposition above with reference only to the stabilizers contained in $G$ follows from the invariance of the 
 $\Sigma^1$-invariants under isomorphisms (Proposition B1.5 in \cite{Strebel}). It is also clear that item $(3)$ is equivalent with asking that
 the restriction of $\chi$ to the actual stabilizers is non zero.

  By Theorem \ref{retractsthm}, if $[\chi] \in \Sigma^2(\Gamma)$ and $\chi |_M = 0$, then $[\chi |_G] \in \Sigma^2(G)$. We can also show that condition 
  $(3)$ of Proposition \ref{suffcond} is necessary.

   \begin{lemma} \label{orbitaX2}
If  $\chi |_{G_{(x,y)}} = 0$, then the monoid $G_{\chi}$ can not have finitely many orbits on $G \cdot (x,y)$.
 \end{lemma}
 
 \begin{proof}
  Suppose that $G \cdot(x,y) = \bigcup_{j=1}^n G_{\chi} \cdot (x_j,y_j)$ and choose $g_1, \ldots, g_n \in G$ such that $(x_j,y_j) = g_j \cdot (x,y)$. 
  Choose $g \in G$ such that 
  \[\chi(g) < min \{\chi(g_j) \hbox{ } |\hbox{ } 1 \leq j \leq n\}.\]
  Since $g \cdot (x,y) \in \bigcup_{j=1}^n G_{\chi} \cdot (x_j,y_j)$, there
  must be some $g_0 \in G_{\chi}$ and $1 \leq j \leq n$ such that $g \cdot (x,y) = g_0 \cdot (x_j,y_j) = g_0 g_j (x,y)$. But then 
  $g^{-1}g_0 g_j \in G_{(x,y)}$, with:
  \[\chi(g^{-1}g_0 g_j) = \chi(g_0) + (\chi(g_j) - \chi(g)) > 0,\]
  so $\chi |_{G_{(x,y)}} \neq 0$.
 \end{proof}

 \begin{prop}  \label{condneces1}
 Let $\Gamma = H \wr_X G$ be a finitely presented wreath product and let $\chi: \Gamma \to \mathbb{R}$ be a non-zero character. Let 
 $M = \oplus_{x \in X} H_x \subseteq \Gamma$ and suppose that $\chi |_M = 0$. If $\chi |_{G_{(x,y)}} = 0$ for some $(x,y) \in X^2$, then $[\chi] \notin \Sigma^2(\Gamma)$. 
 \end{prop}
 
 \begin{proof} We may assume that $[\chi] \in \Sigma^1(\Gamma)$, otherwise there is nothing to do. Thus $[\chi |_G] \in \Sigma^1(G)$ and $\chi |_{G_x} \neq 0$ for all 
 $x \in X$ by Proposition \ref{pro2}.
 
 Let $\Gamma_0 = (\ast_{x \in X} H_x) \rtimes G$ and let $\mathcal{X} \subseteq \Gamma_0$ be a finite generating set. Note that $\Gamma$ is a quotient of
 $\Gamma_0$, so we can consider the following diagram:
 \[ \xymatrix{F(\mathcal{X}) \ar@{->>}[r]^{\pi_0}  \ar@{->>}[rd]_{\pi} & \Gamma_0  \ar@{->>}[d] \\
    & \Gamma.}
 \]
The homomorphism $\pi$ defines presentations for $\Gamma$ with generating set $\mathcal{X}$. We first show that for finite presentations of type 
 $\Gamma = \langle \mathcal{X} | \mathcal{R} \rangle$ (with $ker(\pi) = \langle \mathcal{R} \rangle^{F(\mathcal{X})}$) the complex 
 $Cay(\Gamma;  \langle \mathcal{X} | \mathcal{R} \rangle)_{\chi}$ can not be $1$-connected.
 
 Fix $\langle \mathcal{X} | \mathcal{R} \rangle$ such a presentation. We use the notations $F(\mathcal{X}, \chi)$ and $R(\chi)$ defined in Section 
 \ref{obs}. We want to show that $R(\chi)$ is not finitely generated over $F(\mathcal{X}, \chi)$, from what follows that
 $Cay(\Gamma; \langle \mathcal{X}|\mathcal{R} \rangle)_{\chi}$ is not $1$-connected by Lemma \ref{lemaRfg}.
 
   If $\chi |_{G_{(x,y)}} = 0$ then by Lemma \ref{orbitaX2} we can build a strictly increasing sequence
   \[ I_1 \subsetneq I_2 \subsetneq \ldots \subsetneq I_j \subsetneq \ldots\]
   of $G_{\chi}$-invariant subsets of $X^2$ such that $X^2 = \bigcup_j I_j$.
   
   Let $N$ be the normal subgroup of $\ast_{x \in X} H_x$ such that $M = (\ast_{x \in X} H_x)/N$. Note that $N$ admits an action by 
   $(\ast_{x \in X} H_x) \rtimes G$ (which defines the wreath product $H \wr_X G$). Let $N_j$ be the normal subgroup of  
   $\ast_{x \in x} H_x$ generated by the commutators $[H_x,H_y]$ with $(x,y) \in I_j$. Note that $N_1 \subsetneq N_2 \subsetneq \ldots$, that
   $N = \bigcup_j N_j$ and that each $N_j$ is $(\ast_{x \in X} H_x) \rtimes G_{\chi}$-invariant.
   
   Put \[\Gamma_{j,\chi} = \frac{ \ast_{x \in X} H_x}{N_j} \rtimes G_{\chi}.\]
   This is a well defined monoid under the operation we use to define the semi-direct product. This defines a sequence $\{ \Gamma_{j,\chi} \}_j$ of monoids
   that converges to $\Gamma_{\chi}$.
   
    Now, remember we have chosen $\mathcal{X}$ so that the projection $\pi_0: F(\mathcal{X}) \twoheadrightarrow \Gamma_0$ is well defined. Passing to monoids,
   we obtain a homomorphism
   \[p_0 : F(\mathcal{X}, \chi) \to (\Gamma_0)_{\chi_0},\]
   where $\chi_0$ is the obvious lift of $\chi$ to $\Gamma_0$. From $p_0$ we define
   \[p_j: F(\mathcal{X}, \chi) \to \Gamma_{j,\chi}\]
   for each $j \geq 1$. Let
   \[ R_j = p_j^{-1}(\{1\}).\]
   Notice that $R_j \subsetneq R(\chi)$ for all $j$. Indeed, it is clear that $\chi_0$ and $\chi$ restrict to the same homomorphisms on $G$
   and $G_x$, for all $x \in X$. Also, $\chi_0$ restricts to zero on $\ast_{x \in X} H_x$ by construction. So it follows from Theorem \ref{p1}
   that 
   $[\chi_0] \in \Sigma^1(\Gamma_0)$, since we are assuming that $[\chi] \in \Sigma^1(\Gamma)$. Thus $p_0$ is surjective 
   and any $n \in N \smallsetminus N_j$ defines an element in $R(\chi) \smallsetminus R_j$. Observe further that
   each that $R_j$ is actually a $F(\mathcal{X}, \chi)$-invariant subgroup of $R(\chi)$ and that $\bigcup_j R_j = R(\chi)$. The existence of
   the sequence $\{R_j\}_j$ implies that $R(\chi)$ can not be finitely generated over $F(\mathcal{X},\chi)$.

   For the general case, let $\langle \mathcal{X} | \mathcal{R} \rangle$ be any finite presentation for $\Gamma$ and suppose by contradiction that 
   $Cay(\Gamma; \langle \mathcal{X} | \mathcal{R} \rangle)_{\chi}$ is $1$-connected. From $\langle \mathcal{X} | \mathcal{R} \rangle$ we build
   another finite presentation $\langle \mathcal{X}' | \mathcal{R}' \rangle$ for $\Gamma$ with $\mathcal{X} \subseteq \mathcal{X}'$,
   $\mathcal{R} \subseteq \mathcal{R}'$ and satisfying the previous hypothesis (that is, $\mathcal{X}'$ is actually a generating set for $\Gamma_0$).
   For this, it suffices to add the necessary generators and include the relations that define them in $\Gamma$ in terms of the previous generating set
   $\mathcal{X}$. 
    It may be the case that $Cay(\Gamma; \langle \mathcal{X}' | \mathcal{R}' \rangle)_{\chi}$ is not $1$-connected anymore, but by Lemma 3 in \cite{Renz}, we can
   always enlarge $\mathcal{R}'$ to a (still finite) set $\mathcal{R}''$ so that 
   $Cay(\Gamma; \langle \mathcal{X}' | \mathcal{R}'' \rangle)_{\chi}$ is indeed 
   $1$-connected. This is done by adding the relations of the form $t^{-1}xtw_x^{-1}$,
   as in Theorem \ref{Renz2}. We arrive at a contradiction with the first part of the proof, since
   $\mathcal{X}'$ satisfies the previous hypothesis, that is, $\mathcal{X}'$ can be lifted to a generating set for $\Gamma_0$.
 \end{proof}
 
 The above proposition completes the proof of Theorem \ref{thmC} as stated in the introduction, since its last assertion (when we assume that $H$ has
 infinite abelianization) follows from Theorem \ref{BCKThm}.
 \end{section}
 
 \begin{section}{Applications to twisted conjugacy}
We now derive some consequences of the previous results to twisted conjugacy, more specifically to the study of Reidemeister numbers of automorphisms of 
wreath products. For this we start by considering the Koban invariant $\Omega^1$.
 
Given a finitely generated group $\Gamma$, endow $Hom(\Gamma, \mathbb{R})$ with an inner product structure, so that it makes sense to talk about angles in $S(\Gamma)$.
Denote by $N_{\pi/2}([\chi])$ the open neighborhood of angle $\pi/2$ and centered at $[\chi] \in S(\Gamma)$. Following Koban \cite{Koban}, we can define 
the invariant $\Omega^1(\Gamma)$ in terms of $\Sigma^1(\Gamma)$:
\[\Omega^1(\Gamma) = \{[\chi] \in S(\Gamma) \hbox{ }| \hbox{ } N_{\pi/2}([\chi]) \subseteq \Sigma^1(\Gamma)\}.\]
A proof of the fact that this does not depend on the inner product can be found in the above-mentioned paper, which contains the original definition of the invariant.

Let $\Gamma = H \wr_X G$ be a finitely generated wreath product. With some restrictions on the action by $G$ on $X$, we can obtain nice descriptions of
$\Omega^1(\Gamma)$. Notice that, since the invariant does not depend on the choice of inner product, we can assume that characters $[\chi], [\eta] \in S(\Gamma)$ such that $\chi |_G = 0$ and
$\eta |_M = 0$ are always orthogonal, and this will be done in the proposition below.

\begin{prop} \label{prop1}
Let $\Gamma = H \wr_X G$ be a finitely generated wreath product.
Suppose that  
\[ \Sigma^1(\Gamma) = \{[\chi] \in S(\Gamma) \hbox{ } | \hbox{ } \chi |_M \neq 0\},\]
where $M = \oplus_{x \in x} H_x \subseteq \Gamma$. Then
\[\Omega^1(\Gamma) = \{[\chi] \in S(\Gamma) \hbox{ }| \hbox{ } \chi |_G = 0\}.\]
\end{prop}

\begin{proof}
Let $[\chi] \in S(\Gamma)$ with $\chi |_G = 0$. Clearly $\chi |_M \neq 0$, so $[\chi] \in \Sigma^1(\Gamma)$. Furthermore, if $[\eta] \in N_{\pi/2}([\chi])$,
then $\eta |_M \neq 0$, otherwise $\chi$ and $\eta$ would be orthogonal. So $N_{\pi/2}([\chi]) \subseteq \Sigma^1(\Gamma)$ whenever $\chi |_G = 0$. 
On the other hand, if there were some $[\chi] \in \Omega^1(\Gamma)$ with $\chi |_G \neq 0$, then 
by taking $\eta: \Gamma \to \mathbb{R}$ defined by $\eta |_M = 0$ and $\eta |_G = \chi |_G$, we would have that $[\eta] \in N_{\pi/2}([\chi])$, but 
$[\eta] \notin \Sigma^1(\Gamma)$. 
\end{proof}

For any group $V$, we denote by $V^{ab}$ its abelianization. By Theorem \ref{thmSigma1}, if the $G$-action on $X$ does not contain orbits 
composed by only one element, then many conditions imply the hypothesis on the description of $\Sigma^1(\Gamma)$, such as:
\begin{enumerate}
 \item $(G_x)^{ab}$ is finite for some $x \in X$, or
 \item The set $\{[\chi] \in \Sigma^1(G) \hbox{ } | \hbox{ } \chi |_{G_x} \neq 0\}$ is empty for some $x \in X$. 
\end{enumerate}
This includes the cases where the $G$-action is free (in particular the regular wreath products $\Gamma = H \wr G$) and the case where $\Sigma^1(G) = \emptyset$ .

Recall that the \textit{Reidemeister number} $R(\varphi)$, for a group isomorphism $\varphi: V \to V$, is defined as the number of orbits of the 
$\varphi$-twisted conjugacy action of $V$ on itself. A connection between the invariant $\Omega^1$ and Reidemeister numbers was studied by Koban 
and Wong \cite{KobanWong}. Recall that a character $\chi$ is \textit{discrete} if its image is infinite cyclic.

\begin{theorem}[\cite{KobanWong}, Theorem 4.3] \label{teoKW}
 Let $G$ be a finitely generated group and suppose that $\Omega^1(G)$ contains only discrete characters.
 \begin{enumerate}
  \item[(1)] If $\Omega^1(G)$ contains only one element, then $G$ is of type $R_{\infty}$, that is, $R(\varphi) = \infty$ for all $\varphi \in Aut(G)$.
  \item[(2)] If $\Omega^1(G)$ has exactly two elements, then there is a subgroup $N \subseteq Aut(G)$, with $[Aut(G): N]=2$, such that $R(\varphi) = \infty$ 
        for all $\varphi \in N$.
 \end{enumerate}
\end{theorem}

\begin{crlr} \label{crlr1}
Let $\Gamma = H \wr_X G$ be a finitely generated wreath product and suppose that the $G$-action on $X$ is transitive. Suppose further that
$\Sigma^1(\Gamma)$ is as described in Proposition 
\ref{prop1} and that $H^{ab}$ has torsion-free rank $1$. Then there is a subgroup $N \subseteq Aut(\Gamma)$, with $[Aut(\Gamma):N]=2$, such that $R(\varphi) = \infty$ 
for all $\varphi \in N$.
\end{crlr}

\begin{proof}
By the hypothesis on $H^{ab}$ we have that 
\[ \Omega^1(\Gamma) = \{ [\nu_1], [\nu_2] \},\]
where $\nu_j(G) = 0$, $\nu_1(h) = 1$ and $\nu_2(h)=-1$ for some lift $h \in H$ of a generator for the infinite cyclic factor of $H^{ab}$. It suffices then
to apply part $(2)$ of Theorem \ref{teoKW}.                           
\end{proof}
The applications that we keep in mind are the finitely generated regular wreath products of the form $\mathbb{Z} \wr G$.

Gon\c{c}alves and Kochloukova \cite{GonKoch} exhibited other connections between the $\Sigma$-theory and the property $R_{\infty}$. Below we denote by 
$\Sigma^1(G)^{c}$ the complement of $\Sigma^1(G)$ in $S(G)$, that is, $\Sigma^1(G)^c = S(G) \smallsetminus \Sigma^1(G)$.

\begin{theorem}[\cite{GonKoch}, Corollary 3.4] \label{teoGK1}
Let $G$ be a finitely generated group and suppose that
\[ \Sigma^1(G)^{c} = \{ [\chi_1], \ldots, [\chi_n]\},\]
where $n \geq 1$ and each $\chi_j$ is a discrete character. Then there is a subgroup of finite index $N \subseteq Aut(G)$ such that 
$R(\varphi) = \infty$ for all $\varphi \in N$.
\end{theorem}

\begin{crlr}  \label{crlr2}
Let $\Gamma = H \wr_X G$ be a finitely generated wreath product. Once again, suppose that $\Sigma^1(\Gamma)$ is as described in Proposition \ref{prop1}.
Suppose further that $G^{ab}$ has torsion-free rank $1$. Then there is a subgroup of finite index $N \subseteq Aut(\Gamma)$ such that 
$R(\varphi) = \infty$ for all $\varphi \in N$.
\end{crlr}

\begin{proof}
Under the hypothesis above, we have 
\[\Sigma^1(\Gamma)^{c} = \{[\chi_1], [\chi_2]\},\]
where $\chi_j |_M =0$ and $\chi_1(g) = 1$ and $\chi_2(g) = -1$ for some $g \in G$ whose image in $G^{ab}$ is a generator of the infinite cyclic factor.
Then Theorem \ref{teoGK1} applies.
\end{proof}
This time we can take as an example the regular wreath product $\Gamma = H \wr \mathbb{Z}$.

We note that Gon\c{c}alves and Wong \cite{GonWong} and Taback and Wong \cite{TabackWong} had already obtained some results about the property $R_{\infty}$ for 
regular wreath products of the form $H \wr \mathbb{Z}$, with $H$ abelian or finite. Our results complement theirs in the sense that
it considers other basis groups $H$ and non-regular actions, but here we were limited to talk about Reidemeister numbers of automorphisms
contained in subgroups of finite index in the automorphism group. In the above-mentioned papers, on the other hand, the authors were able to determine 
positively the property $R_{\infty}$ for some choices of $H$.
\end{section}

\section*{Acknowledgements}
The author is supported by PhD grants 2015/22064-6 and 2016/24778-9, from S\~{a}o Paulo Research Foundation (FAPESP).
The author would like to thank his advisor, prof. Dessislava Kochloukova, for the guidance. He also wishes to thank P. H. Kropholler and A. 
Martino, for valuable discussions about their paper, and the anonymous referee, for many suggestions.


\begin{thebibliography}{99}
 
       

  
  
 
 \bibitem{BCK} L. Bartholdi, Y. de Cornulier, D. H. Kochloukova, \emph{Homological finiteness properties of wreath products}. Q. J. Math. 66 (2015), no. 2, 437 - 457.

 \bibitem{BestvinaBrady} M. Bestvina, N. Brady, \emph{Morse theory and finiteness properties of groups}. Invent. Math. 129 (1997), no. 3, 445--470.
 
 \bibitem{Bieri} R. Bieri, \emph{Homological dimension of discrete groups}. Second edition. Queen Mary College Mathematical Notes. Queen Mary College, Department of Pure Mathematics, London, (1981)
 
 \bibitem{BieriGeoghegan} R. Bieri, R. Geoghegan, \emph{Sigma invariants of direct products of groups}. Groups Geom. Dyn. 4 (2010), no. 2, 251–261.
 
 \bibitem{BGK} R. Bieri, R. Geoghegan, D. H. Kochloukova, \emph{The sigma invariants of Thompson's group F}. Groups Geom. Dyn. 4 (2010), no. 2, 263–273. 
 
 \bibitem{BNS} R. Bieri, W. D. Neumann, R. Strebel, \emph{A geometric invariant of discrete groups}. Invent. Math. 90 (1987), no. 3, 451–477.
 
 \bibitem{BieriRenz} R. Bieri, B. Renz, \emph{Valuations on free resolutions and higher geometric invariants of groups}. Comment. Math. Helv. 63(1988), 464 - 497.
 
 \bibitem{BieriStrebel} R. Bieri, R. Strebel, \emph{Valuations and finitely presented metabelian groups}. Proc. London Math. Soc. (3) 41 (1980) 439 - 464.
 
 \bibitem{Bridson}  M. R. Bridson, \emph{The geometry of the word problem}. Invitations to geometry and topology, 29–91, Oxf. Grad. Texts Math., 7, Oxford Univ. Press, Oxford, 2002.
 
 \bibitem{Cornulier}  Y. de Cornulier, \emph{Finitely presented wreath products and double coset decompositions}. Geom. Dedicata 122 (2006), 89 - 108.
 
 \bibitem{Davis} M. W. Davis, \emph{Right-angularity, flag complexes, asphericity}. Geom. Dedicata 159 (2012), 239–262.
 
 \bibitem{DK} M. W. Davis, P. H. Kropholler, \emph{Criteria for asphericity of polyhedral products: corrigenda to "right-angularity, flag complexes, asphericity''}. Geom. Dedicata 179 (2015), 39–44.

 \bibitem{Gehrke} R. Gehrke, \emph{The higher geometric invariants for groups with sufficient commutativity}. Comm. Algebra, 26 (1998), 1097–1115.
 
 \bibitem{GonKoch} D. Gon\c{c}alves, D. H. Kochloukova, \emph{Sigma theory and twisted conjugacy classes}. Pacific J. Math. 247 (2010), no. 2, 335–352.   
 
 \bibitem{GonWong} D. Gon\c{c}alves, P. Wong, \emph{Twisted conjugacy classes in wreath products}. Internat. J. Algebra Comput. 16 (2006), no. 5, 875–886.
 
 \bibitem{Koban} N. Koban, \emph{Controlled topology invariants of translation actions}. Topology Appl. 153 (2006), no. 12, 1975–1993. 

 \bibitem{KobanWong} N. Koban; P. Wong, \emph{A relationship between twisted conjugacy classes and the geometric invariants $\Omega^n$}. Geom. Dedicata 151 (2011), 233–243.
 
 \bibitem{Kochloukova} D. H. Kochloukova, \emph{More about the geometric invariants $\Sigma^m(G)$ and $\Sigma^m(G,\mathbb{Z})$ for groups with normal locally polycyclic-by-finite subgroups}. Math. Proc. Cambridge Philos. Soc. 130 (2001), no. 2, 295–306.
 
 \bibitem{Kochloukova2} D. H. Kochloukova, \emph{On the $\Sigma^2$-invariants of the generalised R. Thompson groups of type F}. J. Algebra 371 (2012), 430–456. 
 
 \bibitem{KrophollerMartino} P. H. Kropholler, A. Martino, \emph{Graph-wreath products and finiteness conditions}. J. Pure Appl. Algebra 220 (2016), no. 1, 422–434.

 \bibitem{MMV} J. Meier, H. Meinert, L. VanWyk, \emph{Higher generation subgroup sets and the $\Sigma$-invariants of graph groups}. Comment. Math. Helv. 73 (1998), no. 1, 22–44. 
 
 \bibitem{MMV2} J. Meier, H. Meinert, L. VanWyk, \emph{On the $\Sigma$-invariants of Artin groups}. Geometric topology and geometric group theory (Milwaukee, WI, 1997). Topology Appl. 110 (2001), no. 1, 71–81. 
 
 \bibitem{MV} J. Meier, L. VanWyk, \emph{The Bieri-Neumann-Strebel invariants for graph groups}. Proc. London Math. Soc. (3) 71 (1995), no. 2, 263–280.
 
 \bibitem{Meinert2} H. Meinert, \emph{The Bieri-Neumann-Strebel invariant for graph products of groups}. J. Pure Appl. Algebra 103 (1995), no. 2, 205–210.
 
 \bibitem{Meinert} H. Meinert, \emph{Actions on 2-complexes and the homotopical invariant $\Sigma^2$ of a group}. J. Pure Appl. Algebra 119 (1997), no. 3, 297–317.
  
 \bibitem{RenzThesis} B. Renz, \emph{Geometrische Invarianten und Endlichkeitseigenschaften von Gruppen}, Dissertation, Frankfurt a.M., (1988).
  
  \bibitem{Renz} B. Renz, \emph{Geometric invariants and HNN-extensions}. Group theory (Singapore, 1987), 465-484, de Gruyter, Berlin, (1989).
 
 \bibitem{Schutz} D. Schütz, \emph{On the direct product conjecture for sigma invariants}. Bull. Lond. Math. Soc. 40 (2008), no. 4, 675–684.
 
 \bibitem{Strebel} R. Strebel, \emph{Notes on the Sigma invariants},  version 2, arXiv:1204.0214 (2012).

 \bibitem{TabackWong} J. Taback, P. Wong, \emph{The geometry of twisted conjugacy classes in wreath products}. Geometry, rigidity, and group actions, 561–587, Chicago Lectures in Math., Univ. Chicago Press, Chicago, IL, 2011.
 
 \bibitem{Zaremsky} M. C. B. Zaremsky, \emph{On the $\Sigma$-Invariants of Generalized Thompson Groups and Houghton Groups}. Int. Math. Res. Not. IMRN (2016), rnw188.
\end{thebibliography}
\end{document}